\documentclass[psamsfonts]{amsart}
\usepackage[dvipsnames]{xcolor}
\usepackage{xcolor}
\usepackage{amsthm}
\usepackage{amsmath}
\usepackage{amssymb}
\usepackage{amscd}
\usepackage[latin2]{inputenc}
\usepackage{t1enc}
\usepackage[mathscr]{eucal}
\usepackage{indentfirst}
\usepackage{graphicx}
\usepackage{graphics}
\usepackage{pict2e}
\usepackage{epic}
\numberwithin{equation}{section}
\usepackage[margin=3cm]{geometry}
\usepackage{epstopdf} 
\usepackage{tikz-cd}
\usepackage{mathtools}
\usepackage{cases}
\usepackage{hyperref}
\usepackage{mathrsfs}
\usepackage{verbatim,amsmath,amsthm,amsfonts,amssymb,latexsym,graphicx,mathtools,extpfeil,color,mathabx}
\usepackage{epstopdf,pinlabel}
\usepackage{amsfonts}
\usepackage[all]{xy}
\usepackage{graphicx}
\usepackage{caption}
\usepackage{subcaption}
\usepackage{tikz} 
\usetikzlibrary{calc,knots}
\DeclareMathAlphabet{\mathpzc}{OT1}{pzc}{m}{it}
\usepackage{slashed}

\theoremstyle{plain}
\newtheorem{Th}{Theorem}[section]
\newtheorem{Lemma}[Th]{Lemma}
\newtheorem{Cor}[Th]{Corollary}
\newtheorem{Prop}[Th]{Proposition}

\newtheorem*{Th*}{Theorem}
\newtheorem*{Cor*}{Corollary}

\theoremstyle{definition}
\newtheorem{Def}[Th]{Definition}

\newtheorem{Rem}[Th]{Remark}

\newtheorem{Ex}[Th]{Example}

\newtheorem*{Acknowledgement*}{Acknowledgement}
\newtheorem{Hominini Disclosure}{Hominini Disclosure}
\newtheorem*{Hominini Disclosure*}{Hominini Disclosure}

\newtheorem*{Conventions*}{Conventions}
\newtheorem*{Def*}{Definition}
\newcommand{\im}{\operatorname{im}}
\newcommand{\tr}{\textup{tr}}

\newcommand{\Cl}{\text{C}\ell}

\newcommand{\la}{\langle}
\newcommand{\ra}{\rangle}
\newcommand{\lp}{\left(}
\newcommand{\rp}{\right)}
\newcommand{\SO}{\text{SO}}

\newcommand{\SPN}{\text{Spin}}
\newcommand{\SU}{\text{SU}}

\newcommand\norm[1]{\left\lVert#1\right\rVert}

\newcommand\spinc{\text{Spin}^{c}}
\newcommand\SPNC{\text{Spin}^{c}}

\newcommand{\ssslash}{\mathbin{/\mkern-5mu/\mkern-5mu/}}

\makeatletter
\def\l@subsection{\@tocline{2}{0pt}{2.5pc}{5pc}{}}
\makeatother

\title[A Fueter Operator for 3/2-spinors]{A Fueter operator for 3/2-spinors}

\author{Ahmad Reza Haj Saeedi Sadegh}
\address{Dartmouth College, Hanover, New Hampshire USA}
\email{ahmadreza.hajsaeedisadegh@dartmouth.edu}

\author{Minh Lam Nguyen}
\address{The University of Texas at Tyler, Tyler, Texas USA}
\email{minhnguyen@uttyler.edu}

\author{Jesus Sanchez Jr}
\address{Clayton State University, Morrow, Georgia USA}
\email{jesussanchez@clayton.edu}

\subjclass[2020]{Primary 53Cxx, 57Rxx, 58Jxx, 57Kxx } 

\begin{document}

	\maketitle
	
	\begin{abstract}
        We show the non-compactness of \textcolor{black}{the moduli space of solutions with a uniform bound on the curvatures} of the monopole equations for $3/2$-spinors on a closed 3-manifold is equivalent to the existence of `3/2-Fueter sections' 
        that are solutions of a non-linear differential equation. These are sections of a fiber bundle 
        whose fiber is a special 4-dimensional submanifold of the hyperk\"ahler manifold of center-framed charged one \textcolor{black}{$SU(3)$}-instantons on $\mathbb{R}^4$. This fiber bundle does not inherit a hyperk\"ahler structure. Our results provide a version of the Haydys correspondence when working with $3/2$-spinors. Furthermore,  although the $3/2$ monopole equations are nonlinear elliptic, the corresponding Fueter operator describing points within the compactification of the $3/2$-monopole equations lacks nonlinear ellipticity. We conclude with a generalization of the notion of Rarita-Schwinger-type operator which provides a framework for understanding the linearization of the $3/2$-Fueter operator and discuss non-ellipticity from this perspective.

		\noindent\textbf{Keywords:} Rarita-Schwinger operator, Seiberg-Witten equations, gauge theory, Fueter sections
  
	\end{abstract}



\section*{Introduction}
\textcolor{black}{There is a generalization of the Seiberg-Witten equations (in dimension three or four) where a spinor bundle is replaced by some fiber bundle whose fibers are hyperk\"ahler manifolds admitting certain $\text{Sp}(1)\times_{\mathbb{Z}/2} \text{U}(1)$-symmetry (e.g., \cite{taubes1999nonlinear, Taubes:2016voz, MR1392667, MR2980921, haydys2015compactness}). In Physics, such a geometric PDE defined on a base manifold (\textit{source manifold}) associated with such a fiber bundle (\textit{target manifold}) is called a gauged-sigma model. The generalized Seiberg-Witten (GSW) equations can be thought of as an analog of the symplectic vortex equations \cite{MR1777853}, and the moduli space of solutions of certain GSW equations is conjectured to carry some important information about manifolds with special holonomy groups (see, e.g., \cite{MR2980921, MR1634503}). Unlike the classical Seiberg-Witten equations, the moduli space is not expected to be compact. Thus, understanding the boundary of the moduli space of GSW equations is an active research problem in mathematical gauge theory (e.g., see \cite{parker2023deformationsmathbbz2harmonicspinors, parker2024gluingmathbbz2harmonicspinors, taubes2024nonconvergentsequencessolutionsmassive}). Moreover, any invariant derived from studying a gauged-sigma model is expected to carry both information about the target manifold and the source manifold. From a low-dimensional topology viewpoint, invariants obtained from GSW may shed some new information about the topology of 3- or 4-dimensional manifolds.}

\textcolor{black}{(Non-linear) Dirac (or Dirac-type) operators are among the key ingredients in defining various (generalized) SW equations. In \cite{nguyen2023pin, sadegh2023threedimensional}, we propose a different generalization of the Seiberg-Witten equations where we replace the Dirac operator with a non-Dirac type operator called the Rarita-Schwinger operator (a definition is given below). Our motivation is to introduce a program defining a topological (or geometrical) invariant of 3- and 4-manifolds using the Rarita-Schwinger operator in the context of Seiberg-Witten-type gauge theory. The Rarita-Schwinger operator was originally introduced to study the dynamics of $3/2$-spin particles; it is one of the few meaningful geometric first-order elliptic differential operators acting on a vector bundle which is not required to be \textit{a priori} a Clifford module (e.g., see \cite{MR1129331, bar2021manifolds}).}

\textcolor{black}{The RS-SW equations share many similar features with generalized Seiberg-Witten equations (specifically, multiple-spinor Seiberg-Witten equations). However, there is already evidence that studying the Rarita-Schwinger Seiberg-Witten (RS-SW)-type equations is interesting. Firstly, in dimension 4, the non-compactness of the moduli space of solutions of  RS-SW equations is directly tied with only topological information of the manifold (see \cite{nguyen2023pin} for more details). Secondly, in dimension 3, under a uniform $L^6$-boundedness of the curvature of the connection involved in the definition of the RS operator, there is a sequence of solutions up to gauge transformations converges weakly to a limiting (twisted) spinor and connection solving degenerate RS-SW equations (see \cite{sadegh2023threedimensional}, also see below for more details). Thirdly, which is the main point of this paper, we extend a Haydys correspondence in this setting, we show that the limiting solutions of the degenerate RS-SW equations correspond to solutions of a non-linear Rarita-Schwinger operator $\mathfrak{Q}$.}

\subsection*{Main result}Let $Y$ be a closed, oriented smooth Riemannian $3$-manifold. Suppose $P_{spin^c}\to Y$ is a $spin^c$ structure on $Y$ and $\slashed{S}:= P_{spin^c} \times_{Spin^c(3)} \mathbb{H}$ is denoted by the associated spinor bundle over $Y$. $\slashed{S}$ is a Clifford module; and thus there is a Dirac operator $D_A$ on it. $D_A$ is determined by 
\begin{itemize}
    \item a Clifford multiplication induced by the Riemannian metric on $Y$,
    \item the covariant derivative $\nabla_A$ induced by the Levi-Civita connection on $Y$ and a $U(1)$-connection $A \in \mathcal{A}(\det P_{spin^c})$ on $\det P_{spin^c}=P_{spin^c}/Spin(3)$.
\end{itemize}
The Dirac operator $D_A$ acts on sections of $\slashed{S}$, which are referred to as $1/2$-spinors in Physics. There is also a cousin to the Dirac operator called the \textit{Rarita-Schwinger (RS) operator} \cite{PhysRev.60.61, homma2019kernel, bar2021manifolds}. Associated to a $spin^c$ structure on $Y$, a RS operator is also defined based on a choice of a $U(1)$-connection on the determinant line bundle of the $spin^c$ structure. Very briefly, a RS operator $Q_A$ is defined by
$$Q_A = \pi_{3/2} \circ D^{TY}_A|_{\Gamma(X, ker\, c)},$$
where 
\begin{itemize}
    \item $ker\, c$ denotes the sub-bundle of $T^*Y \otimes \slashed{S}$ that is the kernel of the Clifford multiplication $c: T^* Y \otimes \slashed{S} \to \slashed{S}$,
    \item $D^{TY}_A$ is a Dirac operator on $T^*Y \otimes \slashed{S}$,
    \item $\pi_{3/2}$ is the orthogonal projection of $T^*Y \otimes \slashed{S} \to ker\,c$.
\end{itemize}

The operator $Q_A$ acts on sections of $ker\,c$, which are referred to as \textit{$3/2$-spinors}. $Q_A$ is the first example of many higher-spin Dirac operators associated with a $spin^c$ structure on $Y$. Having set up the terminologies, the Rarita-Schwinger-Seiberg-Witten (RS-SW) equations are a system of geometric PDEs on $Y$ that look for unknowns $(A, \psi) \in \mathcal{A}(\det P_{spin^c}) \times \Gamma(Y, ker\, c)$ satisfying
\begin{equation}\label{eq: RS-SW}
    \begin{cases}
        Q_A \psi = 0,\\
        \star_3 F_{A} - \mu(\psi)=0.
    \end{cases}
\end{equation}
where $\mu: ker\,c \subset \Gamma(Y, T^*Y \otimes \slashed{S}) \to \Gamma\lp i\mathfrak{su}(\slashed{S})\rp$ is a quadratic map that maps each $3/2$-spinor $\psi$ to the traceless part of the endomorphism $\psi \psi^*$ on $\slashed{S}$. 

By blowing-up along the locus of the reducible solutions $\psi\equiv 0$ of \eqref{eq: RS-SW}, we obtain the equations with an extra unknown $\epsilon \in (0,\infty)$
\begin{equation}\label{eq: RS-SWblowup}
    \begin{cases}
        \textcolor{black}{\norm{\psi}_{L^4}=1},\\
        Q_A \psi = 0,\\
        \epsilon^2 \star_3 F_A - \mu(\psi) = 0.
    \end{cases}
\end{equation}
\textcolor{black}{The reason for blowing-up using $L^4$-norm is given in Remark \ref{RemL^4} below.}

The equation \eqref{eq: RS-SW} has an abelian gauge symmetry given by $\mathcal{G} = Maps(Y, S^1)$. The moduli space is defined to be the solution space of \eqref{eq: RS-SW} mod $\mathcal{G}$. Note that the solutions to \eqref{eq: RS-SW} can also be thought of as critical points of a certain modified Chern-Simons-Dirac functional $\mathcal{L}^{RS}$, where \eqref{eq: RS-SW} is interpreted as the minimizing condition for $\mathcal{L}^{RS}$ (see \cite{sadegh2023threedimensional} for more details). 

In dimension four, the  RS-SW equations were first introduced by the second-named author in \cite{nguyen2023pin}. The moduli space of the $4$-D  RS-SW equations is known to be \textit{not} compact in general--a feature that is different from the classical Seiberg-Witten theory. There is a topological obstruction in terms of an inequality relating the second Betti number $b_2$ and the signature of the four manifolds that guarantees the non-compactness of the moduli space. Thus, on those four manifolds that satisfy such topological conditions, we are guaranteed to have non-trivial solutions of the  RS-SW equations. 

The story is a bit different in dimension three. In \cite{sadegh2023threedimensional}, we studied the behavior of the convergence of solutions of \eqref{eq: RS-SW} after the blow-up procedure. In particular, we proved that if $\{(A_n, \psi_n, \epsilon_n)\}$ is a sequence of solutions of \eqref{eq: RS-SWblowup} such that $F_{A_n}$ are uniformly bounded in $L^6$ and $\epsilon_n$ gets arbitrarily small, then away from a closed nowhere-dense subset $Z\subset Y$, after passing through a subsequence and up to gauge transformations, $A_n$ converges weakly to $A$ in $L^2_{1,loc}$, $\psi_n$ converges weakly to $\psi$ in $L^2_{2,loc}$, and $(A,\psi)$ must satisfy the following degenerate situation
\begin{equation}\label{eq:degenerate}
Q_A \psi = 0, \quad \mu(\psi) =0.
\end{equation}

\begin{Rem}\label{RemL^4}
    \textcolor{black}{We give a brief comment on why we blow up \eqref{eq: RS-SW} using the $L^4$-norm as opposed to the $L^2$-norm that usually appears in ordinary generalized Seiberg-Witten theory. Unlike the Wietzenb\"ock formula of the Dirac operator, $Q_A^2$ contains a second-order term $P_AP^*_A$, where $P_A := \pi_{3/2}\circ \nabla_A$. A common strategy to prove such a convergence result stated above is starting with Green's integration by parts formula, where one has to control $\norm{P^*_A\psi}_{L^2}$. Essentially, the divergence of $\psi$ can be controlled by the curvature $F_A$ via gauge-fixing if $\psi$ is assumed to have $L^4$-unit norm by a standard elliptic estimate applied to the Rarita-Schwinger operator in dimension 3 (see Section 3 of \cite{sadegh2023threedimensional} for more details).}
\end{Rem}

The limiting solutions of the moduli space are called \textit{$3/2$-Fueter sections}. Formally, they solve \eqref{eq: RS-SWblowup} by setting $\epsilon = 0$. Then one of the consequences of the above statement is that the existence of $3/2$-Fueter sections is an obstruction to the compactness of the moduli space of \eqref{eq: RS-SW}. Indirectly, if one can show that \eqref{eq: RS-SWblowup} with $\epsilon = 0$ has no solution in certain situations, then potentially the moduli space of \eqref{eq: RS-SW} would be compact!

In this paper, we study the degenerate situation of \eqref{eq: RS-SW}. Viewing the Clifford module $T^* Y\otimes \slashed{S}$ as a (linear) hyperk\"ahler fiber bundle over $Y$ and $\mu$ as an associated hyperk\"ahler moment \textcolor{black}{map}, the hyperk\"ahler reduction $\mu^{-1}(0)/S^1$ in turn is another hyperk\"ahler fiber bundle over $Y$ \cite{hitchin1987hyperkahler}. We consider a subbundle of $\mu^{-1}(0)/S^1$ by intersecting the entire total space with $ker\,c$. Denote such a subbundle by $\mathscr{W}_0$ \textcolor{black}{whose fiber dimension is $4$ (ref. Proposition \ref{p:w_0ismanifold})}. Note that $\mathscr{W}_0$ no longer necessarily inherits a hyperk\"ahler structure. However, there is a one-to-one correspondence between solutions of \eqref{eq:degenerate} and solutions of a certain non-linear differential operator \textcolor{black}{$\mathfrak{Q}$} defined on $\mathscr{W}_0$. In particular, the main result of this paper shows that

\begin{Th*}\label{maintheorem}
    For any gauge equivalence class of solutions $[(A,\psi)]$ of \eqref{eq:degenerate}, there is a solution 
     $\Psi\in\Gamma(\mathscr{W}_0)$ of the $3/2$-Fueter equation
     \begin{equation}\label{eq:3:2Fueterequation}
         \color{black}\mathfrak{Q}\Psi=0.
     \end{equation}
    Conversely, for any solution $\Psi$ of \eqref{eq:3:2Fueterequation}, there exists a $\SPNC(3)$-structure $P_{\SPNC}\to Y$ with a connection $A$
     on its determinant line bundle, a section  $\psi\in \Gamma(\mathscr{W})$ where $\mathscr{W}:=P_{\SPNC}\times_{\SPNC(3)}W$ such that the pair $(A,\psi)$ satisfies the equation \eqref{eq:degenerate}. Moreover, any two solutions $\lp A_1,\psi_1\rp$ and $\lp A_2,\psi_2\rp$ of \eqref{eq:degenerate} which correspond to $\Psi$ must be gauge equivalent. Furthermore, $\mathfrak{Q}$ is a non-linear operator that is not elliptic.
\end{Th*}

We call \textcolor{black}{$\mathfrak{Q}$} a \textit{3/2-Fueter operator}. To the best of our knowledge, this is the first time such an operator has been defined in the literature. Its construction mirrors that of the Fueter operator, which can be thought of as a generalized (non-linear) Dirac operator defined on a hyperk\"ahler fiber bundle over $Y$ associated with a \textcolor{black}{$\SPNC$} structure (see \cite{MR2980921, haydys2015compactness, taubes1999nonlinear}). Whenever there is a Dirac operator, there is also an RS operator. Following this philosophy, it is not unexpected to have a counterpart of the Fueter operator for the $3/2$-spinors. Our 3/2-Fueter operator \textcolor{black}{$\mathfrak{Q}$} fulfills exactly this role. In other words, \textcolor{black}{$\mathfrak{Q}$} is the gauged-sigma model for the RS operator. 

One of the main technical difficulties one faces when working with $\mathscr{W}_0$ is that it is defined using both Clifford contraction $c$ and the quadratic map $\mu$, and these maps do not interact cleanly with one another. This creates several problems when one wants to ascertain basic properties of the objects at play e.g. 
\begin{enumerate}
    \item[(I)] smoothness of $\mathscr{W}_0$ (cf. Proposition \ref{p:w_0ismanifold});
    \item[(II)] analytic properties of the linearization of $\mathfrak{Q}$ (cf. Theorem \ref{ellipticity}).
\end{enumerate}

In the hyperk\"ahler framework, the quaternionic structure does most of the work. In our framework we needed to utilize a different structure, namely that there is a background $Spin^c\lp 3\rp$ action which preserves the quadratic map and the bundle $\ker c\rightarrow Y$. Perhaps the result which is most exotic in our main theorem is the fact that the linearization of $\mathfrak{Q}$ is \emph{not} a nonlinear elliptic operator. Following the usual expectation, one may have conjectured that for a general nonlinear elliptic PDE, the compactification of its solution space would consist of adding in solutions to an appropriate limiting nonlinear PDE which maintained ellipticity. In fact, prior to the writing of this paper, we believed that such would be the case. This has been the case for many nonlinear elliptic PDEs' within the literature \cite{MR1745014, MR1634503, haydys2012gauge, haydys2015compactness, esfahani2023monopolefueterfloerhomology}; but does not seem to be so for the $3/2$-monopole equations on a $3$-manifold. 

\begin{Rem}
    One should think of the main theorem above as a so-called \textit{Haydys correspondence} for 3/2-spinors (see \cite{MR2980921, haydys2015compactness}).
\end{Rem}

\begin{Rem}
    We have a rather explicit local description of the fibers of $\mathscr{W}_0$ (ref. Proposition \ref{p:w_0ismanifold}). The fibers are $4$-dimensional manifolds and can be viewed as a variety, which means that the total space $\mathscr{W}_0$ is $7$-dimensional. It would be interesting to know more about the topology and geometry of $\mathscr{W}_0$.
\end{Rem}

\textcolor{black}{\begin{Rem}\label{r:broaderscope}
    In this article, the construction of the fiber bundle $\mathscr{W}_0$ and the 3/2-Fueter operator is obtained from applying the fiberwise hyperk\"ahler reduction to the subbundle of $3/2$-spinors associated with a $\SPN^c$-structure on $Y$. One could generalize from the $\SPN^c$-structure to a more general \emph{Clifford module}, that is a complex (hermitian) vector bundle $\mathscr{E}\to Y$ that carries a fiberwise Clifford action of the tangent bundle $TY$, i.e., a bundle map $c:TY\otimes \mathscr{E}\to \mathscr{E}$ with a fixed \textup{Clifford connection}. In this situation, one can consider the subbundle $\mathscr{W}:=\ker(c)$ which carries a similar contraction of the Dirac operator, a generalized Rarita-Schwinger operator $Q:\Gamma(\mathscr{W})\to \Gamma(\mathscr{W})$. Assume there is a fiberwise moment map $\mu:\mathscr{E}\to i\mathfrak{su}(2)$ that is invariant under the $U(1)$-action (induced from the complex structure). If $0\in i\mathfrak{su}(2)$ is a regular value for  both $\mu$ and $\mu|_{\mathscr{W}}$ then one can apply the hyperk\"ahler reductions to fibers of $\mathscr{W}$ to obtain a fiber bundles $\mathscr{E}_0\to Y$ and $\mathscr{W}_0\to Y$. The bundle $\mathscr{E}_0$ is a hyperk\"ahler bundle which carries a Fueter operator. It would be interesting to see if one can mimic the method in this article to obtain a 3/2-Fueter operator $\mathfrak{Q}$ on $\mathscr{W}_0$. 
\end{Rem}}

The  RS-SW equations can be of interest to high-energy Physics. From a physical perspective, the equations describe a minimally coupled $U(1)$ gauge field with spin $3/2$-charged matter. In Physics, spin $3/2$-fields are usually considered to be a part of the supergravity multiplet, but those would not be charged under $U(1)$ gauge symmetry. Thus, solutions of \eqref{eq: RS-SW} should be referred to as spin $3/2$-charged matter. However, it is possible to swap out $U(1)$ with some other compact Lie group $G$ to account for the twisted supergravity phenomenon and derive a similar behavior of convergence result as in \cite{sadegh2023threedimensional} and a version of Haydys correspondence analogous to the main result of this paper. This circle of ideas will be addressed elsewhere.

As a concluding remark for this subsection, we state the main motivation for us to study \eqref{eq: RS-SW}. Our larger goal is to derive a new $3$-manifold invariant using RS operator. The equations we study fall largely within a framework of Lagrangian QFTs. Thus, one could either define a numerical invariant via Rozansky-Witten's approach by taking a formal path integral of $\exp(\text{const}\cdot \mathcal{L}^{RS})$ against the moduli space of \eqref{eq: RS-SW}, or define a homological invariant via a certain Floer theory approach that should categorify the previously mentioned numerical invariant. The results in our previous paper \cite{sadegh2023threedimensional} and in this current one serve as a first step toward this program.


\subsection*{Organization of the paper}
In Section \ref{s:preliminaries} of the paper, we discuss background materials. We begin by recalling the notions of connection and covariant derivative on fiber bundles, 3/2-spinors, and hyperk\"ahler structure and its reduction. We then discuss the notion of hyperk\"ahler bundles and their associated generalized Dirac operators.  

In Section \ref{s:aquaternionicmodulispace}, we introduce an 8-dimensional hyperk\"ahler manifold with an action of \textcolor{black}{$\SPNC\lp 3\rp$} and a special invariant 4-dimensional submanifold $W_0$. We study this 4-manifold in more detail in the subsequent subsections.  We show the existence of a canonical normal bundle over this manifold. We then introduce a nonlinear differential operator on this bundle that is called the \textit{3/2-Fueter operator}.

In Subsections \ref{sec:Haydyscorrespondence} and \ref{sec:ellipticity}, we prove the paper's main result by dividing it into two parts. In the first part, we prove that the existence of solutions to the degeneracy equations \eqref{eq:degenerate} is equivalent to the existence of solutions to the 3/2-Fueter operator. In the second part, we prove that the 3/2-Fueter equation is a non-elliptic equation. 

In Section \ref{sec:generalrsop}, we give a definition of a generalized Rarita-Schwinger-type operator, which is a general formalism of the linearized $3/2$-Fueter operator. There, we also discuss a general (non)ellipticity result regarding these operators. Finally, we give an alternative proof of the non-ellipticity of the linearization of $\mathfrak Q$ based on the general (non)ellipticity result developed there.

\begin{Acknowledgement*}
   We would like to thank \textcolor{black}{Cumrun Vafa} and Sergei Gukov for a discussion about the broad physical framework that our work could fit into. The second-named author would like to thank Gregory Parker for some of the discussions and suggestions regarding the construction of the 3/2-Fueter operator \textcolor{black}{$\mathfrak{Q}$}. Parts of this work were carried out while the authors were at the Instantons and Foams conference at MIT in May 2023. The conference was supported by the Simons Collaboration on New Structures in Low-dimensional Topology. The first-named author thanks the Max Planck Institute for Mathematics for its support and hospitality during summer 2026, when part of this work was carried out.
\end{Acknowledgement*}
\begin{Hominini Disclosure*}
    The main ideas of this article, the theorems, and the proofs were all developed solely by homo sapiens.
\end{Hominini Disclosure*}
\begin{Conventions*}
Throughout our article we freely utilize the natural isomorphism $U(2)\simeq Spin^c\lp 3\rp$. 
\end{Conventions*}

\section{Preliminaries}\label{s:preliminaries}
\subsection{Connections on fiber bundles}
We first recall some basic facts about \textcolor{black}{Ehresmann} connections on a fiber bundle. Consider a fiber \textcolor{black}{bundle} $\pi:\mathscr{M}\to Y$. The \emph{vertical bundle} $\mathcal{V}\mathscr{M}\to\mathscr{M}$ is defined as the kernel of the differential map $d\pi:T\mathscr{M}\to TY$. A choice of (\textcolor{black}{Ehresmann}) connection on $\mathscr{M}\to Y$ determines a bundle decomposition $T\mathscr{M}=\mathcal{V}\mathscr{M}\oplus \mathcal{H}\mathscr{M}$ where \textcolor{black}{we call} $\mathcal{H}\mathscr{M}$ the horizontal bundle whose fibers naturally identify with tangent spaces of $Y$, i.e., for all $p\in \mathscr{M}$:
\[\color{black}\mathcal{H}\mathscr{M}|_p\buildrel{d\pi}\over\simeq T_{\pi(p)}Y.\]

The covariant derivative (associated with the connection) of a section of the fiber bundle is defined as follows. Let $\phi\in\Gamma(\mathscr{M})$ be a section, i.e., a map $\phi:Y\to \mathscr{M}$ such that $\pi\circ\phi=\textup{id}_Y.$ Taking a differential gives a bundle map $d\phi:TY\to T\mathscr{M}.$ Composing with the projection map $T\mathscr{M}\to \mathcal{V}\mathscr{M}$ gives the covariant derivative of the section:
\[\nabla\phi:TY\to \mathcal{V}\mathscr{M}\]
and equivalently, a $\phi^*\mathcal{V}\mathscr{M}$-valued one-form
\begin{equation}\label{eq:connectiononeform}
    \nabla\phi\in \Omega^1(Y,\phi^*\mathcal{V}\mathscr{M}).
\end{equation}

\begin{Rem}
    When $\pi:\mathscr{M}\to Y$ is a vector bundle, one has a natural isomorphism $\mathscr{M}\simeq \phi^*\mathcal{V}\mathscr{M}$ and hence the covariant derivative above is an $\mathscr{M}$-valued one-form.

 \end{Rem}

When $P\to Y$ is a principal $G$-bundle, by a connection, we mean an equivariant decomposition $TP=\mathcal{V}P\oplus \mathcal{H}P$. Fixing a connection on the principal bundle $P\to Y$
    gives a covariant derivative on any associated bundle $\mathscr{M}=P\times_{G}M$
    where $M$ is a \textcolor{black}{manifold with an action of $G$}; any section $\phi\in \Gamma(\mathscr{M})$ can be considered as an equivariant function $\phi:P\to M$. The restriction of the differential $d\phi:TP\to TM$ to the horizontal subbundle $\mathcal{H}P$ gives \textcolor{black}{a $G$-equivariant} bundle map
    \[d\phi|_{\mathcal{H}P}:\mathcal{H}P\to TM.\] 
    \textcolor{black}{which is equivalent to a map $\nabla\phi:TY\to P\times_G TM$. Now using the canonical isomorphism $\mathcal{V}\mathscr{M}\simeq P\times_G TM$ we have obtained the bundle map}
    \[\color{black}\nabla\phi:TY\to \mathcal{V}\mathscr{M},\]
    and equivalently, a vector bundle valued one-form
    \[\nabla\phi\in \Omega^1(Y,\phi^*\mathcal{V}\mathscr{M}).\]

\subsection{3/2-spinors}\label{ss:3/2-spionors}In this article, $\mathbb{H}$ denotes the division algebra of quaternions with its natural inner product, and $\text{Im}\mathbb{H}$ denotes its subspace of pure imaginary elements.
Let $\mathbb{H}^{\pm}$ be the positive and negative half-spin representations of the Clifford algebra $\Cl(4)$ with their natural identification with quaternions $\mathbb{H}$. Both $\mathbb{H}^{\pm}$ give the two spin representations of the Clifford algebra $\Cl(3)$; however, they are equivalent as the representations of the spin group $\SPN(3)\subset \Cl(3)$.

\textcolor{black}{The Clifford representation of $\Cl(3)=\Cl(\text{Im}\mathbb{H})$ on $\mathbb{H}$ is induced from the standard left action of the imaginary quaternions on the quaternions. It is easy to see that this representation is complex.}
The Clifford action on the spinors $\mathbb{H}$ then gives a linear map \textcolor{black}{(identifying $\mathbb{C}\otimes_{\mathbb{R}}\text{Im}\mathbb{H}\simeq \mathbb{C}^3$)}
\begin{align*}
\color{black}\widehat{c}:\textup{Hom}_{\mathbb{C}}(\mathbb{C}\otimes_{\mathbb{R}}\text{Im}\mathbb{H},\mathbb{H})&\to\mathbb{H}\\
\lp z\otimes_{\mathbb{R}} q\rp^*\otimes_{\mathbb{C}} w&\mapsto z\cdot q\cdot w.
\end{align*}
where $\lp z\otimes_{\mathbb R}q\rp^*$ denotes the Hermitian dual. The kernel of the above map is what we call the space of \emph{3/2-spinors}:
\[W:=\ker(\widehat{c}).\]
Observe that $\text{Im}\mathbb{H}$ admits an action by $\text{SU}\lp 2\rp\simeq \mathbb{S}^3$ via the adjoint representation, and $\mathbb{H}$ admits the standard left action of $\text{SU}\lp 2\rp$. Thus, $\text{Hom}_{\mathbb{C}}\lp\mathbb{C}\otimes_{\mathbb{R}}\text{Im}\mathbb{H},\mathbb{H}\rp$ also admits an action of $SU(2)$. The standard hermitian structure on $\text{Hom}_{\mathbb{C}}\lp\mathbb{C}\otimes_{\mathbb{R}}\text{Im}\mathbb{H},\mathbb{H}\rp$ given by $\la
\phi,\psi\ra:=\frac{1}{2}\text{tr}\lp\phi\psi^*\rp$ is $\text{SU}\lp2\rp$-invariant. This hermitian structure coincides with the tensor hermitian structure on $(\mathbb{C}\otimes_{\mathbb{R}}\text{Im}\mathbb{H})^*\otimes_{\mathbb{C}}\mathbb{H}\simeq \textup{Hom}(\mathbb{C}\otimes_{\mathbb{R}}\text{Im}\mathbb{H},\mathbb{H})$. One can check that in this case that the map $\widehat{c}$ is $\text{SU}\lp 2\rp$-equivariant. Thus $W$ is also acted on by $\text{SU}\lp 2\rp$ with weight $m=3/2$ and
the orthogonal complement of $W$ in $\textup{Hom}(\mathbb{C}\otimes_{\mathbb{R}}\text{Im}\mathbb{H},\mathbb{H})$ is a copy of spinor space $\mathbb{H}$. Indeed, consider the map
\begin{flalign}\label{eq:iotamap}
    \iota:\mathbb{H}&\to \textup{Hom}(\mathbb{C}\otimes_{\mathbb{R}}\text{Im}\mathbb{H},\mathbb{H})\\
    s&\rightarrow \int_{\mathbb{S}^3}q^*\otimes qsdV\lp q\rp\notag
\end{flalign}
The map $\iota$ is $\text{SU}\lp 2\rp$-equivariant and provides 
the $\text{SU}\lp 2\rp$-invariant orthogonal decomposition
\[\textup{Hom}(\mathbb{C}\otimes_{\mathbb{R}}\text{Im}\mathbb{H},\mathbb{H})= \iota(\mathbb{H})\oplus W.\]

We also put a Riemannian structure on $M^+:=\textup{Hom}(\mathbb{C}\otimes_{\mathbb{R}}\text{Im}\mathbb{H},\mathbb{H})$, by putting the inner product on the tangent spaces $T_{\phi}M^+\simeq M^+$ given by
\[(A,B):=\frac{1}{2}\textup{Re}(\tr(AB^*)).\]

\subsection{\textcolor{black}{Hyperk\"ahler} manifolds}\label{sec:hyperkahlermanifoldsfuetersections}
Let $(M,g,I,J,K)$ be a hyperk\"ahler manifold that is a Riemannian manifold $(M,g)$ with three integrable almost complex structures on the tangent bundle that are $g$-\textcolor{black}{K\"ahler} and satisfy the quaternionic relations:
\[I^2=J^2=K^2=IJK=-1.\]
One has three K\"ahler forms
\[\omega_I(u,v):=g(Iu,v), \ \ \omega_J(u,v):=g(Ju,v), \ \ \omega_K(u,v):=g(Ku,v).\]
\begin{Prop}\cite{hitchin1987hyperkahler}
    With respect to the complex structure $I$, one has the holomorphic symplectic form
    \[\Omega=\omega_J+i\omega_K.\]
    Conversely, any K\"ahler manifold with a holomorphic symplectic form is hyperk\"ahler. 
\end{Prop}

The 2-sphere of complex structures 
\begin{equation}\label{eq:complexstructuresphere}
    \mathfrak{b}(M,g)=\{aI+bJ+cK: a^2+b^2+c^2=1\}
\end{equation}
consists of integrable almost complex structures on $M$ that are K\"ahler with respect to the metric $g$. When the Riemannian metric is fixed we drop the reference to the metric for brevity of the notation. Hence for any oriented orthonormal basis $u_1,u_2,u_3$ of $\mathbb{R}^3$, the triple $(I',J',K')$ with \[I':=u_1\cdot (I,J,K),\ J':=u_2\cdot (I,J,K),\ K'=u_3\cdot (I,J,K),\]
give a hyperk\"ahler structure on $(M,g)$ equivalent to the original one (obtained by a rotation). 

The most basic examples of a hyperk\"ahler manifolds are quaternionic vector spaces, including $\mathbb{H}$ or $\textup{End}_{\mathbb{C}}(E,\mathbb{H})$ where $E$ is a complex hermitian vector space. One obtains more such examples via hyperk\"ahler quotients.

\subsection{Hyperk\"ahler reduction}
Consider a hyperk\"ahler manifold $(M,g,I,J,K)$ carrying an action of a compact Lie group $G$ \textcolor{black}{preserving} the symplectic forms $\omega_I,\omega_J,\omega_K$. Such an action of $G$ on $M$ is called a \textit{hyperk\"ahler action}.

For any complex structure $S \in \{I, J, K\}$ and any $\xi \in \mathfrak{g}$, by Cartan's magic formula we have
$$0=L_{X^M_\xi} \omega_S = d(\iota_{X^M_\xi}\omega_S) + \iota_{X^M_\xi}(d\omega_S),$$
where $\iota_v$ denotes the usual contraction operator of a differential form by a vector field on $M$, and $X^M_\xi$ denotes the \textcolor{black}{Killing} vector field on $M$ associated with $\xi$. Since $\omega_S$ is a symplectic form, the above identity tells us that $\iota_{X^M_\xi}\omega_S$ is a closed $1$-form. Furthermore, if we assume that $H^1_{dR}(M;\mathbb{R})$ is trivial, then every closed $1$-form is exact. Thus, there exists a function
$$\mu^{S}_{X^M_\xi}: M\to \mathbb{R}, \quad d\mu^{S}_{X^M_\xi} = \iota_{X^M_\xi}\omega_S.$$
For this, we may define (\textcolor{black}{since $G$ is compact we may assume that there is an $Ad$-invariant metric on the Lie algebra $\mathfrak{g}$ of $G$}) a map $\mu^S: M \to \mathfrak{g}^*$ where 
$$\la \mu^S(m), \xi\ra = \mu^S_{X^M_\xi}(m).$$
We can combine each of the moment maps defined above as
$$\color{black}\mu = \mu^I i + \mu^J j + \mu^K k : M \to \mathfrak{g}^* \otimes_{\mathbb{R}} \text{Im}\mathbb{H} \cong \mathfrak{g}^* \otimes_{\mathbb{R}} \mathfrak{su}(2).$$\textcolor{black}{Here, $\{i, j, k\}$ denotes the standard bases of the purely imaginary quaternions $\text{Im}\mathbb{H}$}. Note that by construction, such a map $\mu$ is $G$-equivariant. This prompts the following definition.

\begin{Def}\label{hyperkahlermoment}
    A \textit{hyperk\"ahler moment map} of a hyperk\"ahler action $G\curvearrowright M$ is a map $\mu: M\to \mathfrak{g}^* \otimes_{\mathbb{R}} \mathfrak{su}(2)$ such that
    \begin{enumerate}
        \item $\mu$ is $G$-equivariant,
        \item \textcolor{black}{$d\la\mu,\xi\ra = \iota_{X^M_\xi}\omega$, for all $\xi \in \mathfrak{g}$ and $\omega$ is the hyperk\"ahler form associated with the hyperk\"ahler structure on $M$. Here we view $\omega$ as the triple of hyperk\"ahler forms previously defined, now interpreted as an $\mathfrak{su}(2)$-valued form.}
    \end{enumerate}
    If a hyperk\"ahler $G$-action possesses a hyperk\"ahler moment map $\mu$, we say that $G\curvearrowright M$ is a \textit{tri-Hamiltonian action}.
\end{Def}

\textcolor{black}{\begin{Rem}\label{r:standardmu}
    The notation $\mu$ used to describe a hyperk\"ahler moment map on $G$-hyperk\"ahler manifold $M$ is standard. However, in the situation where $M := \text{Hom}_{\mathbb{C}}(E, \mathbb{H})$ and $G= U(1)$, where $E$ is some $n$-dimensional hermitian vector space, we consider a particular moment map that coincides with the quadratic map that appears in (generalized) Seiberg-Witten theory (cf. Example \ref{ex:mainhyperkahlerquotient}). Thus, with abuse of notation, we use the same label $\mu$ for such a moment map on $\text{Hom}_{\mathbb{C}}(E, \mathbb{H})$. Of course, this is a very specific phenomenon that happens for certain \textit{linear} hyperk\"ahler manifolds, and not all hyperk\"ahler $U(1)$-manifolds will have a moment map arise in this manner.
\end{Rem}}

To get a new hyperk\"ahler manifold from an old one, we suppose that there is a tri-Hamiltonian action of $G$ which is free and consider the associated hyperk\"ahler moment map
\[\mu:M\to \mathfrak{g}^*\otimes_{\mathbb{R}}\mathfrak{su}(2).\]
\begin{Th}\cite{hitchin1987hyperkahler}
    If $c\in \mathfrak{g}^*\otimes_{\mathbb{R}}\mathfrak{su}(2)$ is an invariant regular value, then the quotient $\mu^{-1}(c)/G$ is a hyperk\"ahler manifold of dimension $\dim(M)-4\dim(G).$
\end{Th}
When the choice of the invariant regular value $c$ is clear, we simply denote by $M\ssslash G$ the hyperk\"ahler reduction $\mu^{-1}(c)/G$. Denote by $\mathcal{L}\subset TM$ the subbundle spanned by the Killing vector field of the action of $G$. So, the fibers of $\mathcal{L}$ are isomorphic to the Lie algebra $\mathfrak{g}$. We will need the following lemma in the subsequent sections:
\begin{Lemma}\label{l:decompositionofinverseimage}\cite{hitchin1987hyperkahler}
    One has the following $G$-equivariant orthogonal decomposition:
    $$TM|_{\mu^{-1}(c)}=\mathcal{H}\oplus \mathcal{L}\oplus I\mathcal{L}\oplus J\mathcal{L}\oplus K\mathcal{L}$$
    where $\mathcal{H}$ is a vector bundle whose fiber is isometrically isomorphic to the tangent spaces of $M\ssslash G$ under the quotient map $\tau:M\to M\ssslash G$ and one also has
    \[T\mu^{-1}(c)=\mathcal{H}\oplus\mathcal{L}.\]
\end{Lemma}

\begin{Ex}\label{ex:mainhyperkahlerquotient}
The hyperk\"ahler manifold $M:=\textup{Hom}_{\mathbb{C}}(E,\mathbb{H})\cong \mathbb{H}\otimes_{\mathbb{C}} E^*$, where $E$ is an $n$-dimensional hermitian vector space, comes with a $U(1)$-moment map associated with a $U(1)\curvearrowright M$ induced by the left multiplication of $U(1)$ on $\mathbb{H}$
\[\mu:\textup{Hom}(E,\mathbb{H})\to \mathfrak{u}(1)\otimes_{\mathbb{R}}\mathfrak{su}(2)\]
\begin{equation}\label{eq:basicmomentmap}
    \color{black}\psi\mapsto \psi\psi^*-\frac{1}{2}|\psi|^2\textup{id}_{\mathbb{H}}.
\end{equation}
\textcolor{black}{Here, we view $\psi^*: \mathbb{H} \to E$ so that $\psi \psi^* : \mathbb{H} \to \mathbb{H}$, which is self-adjoint endomorphism of $\mathbb{H}$. By removing the trace part $\psi\psi^*$, $\mu(\psi)$ a self-adjoint traceless endomorphism of $\mathbb{H} \cong \mathbb{C}^2$. Note that $\mathfrak{su}(2)$ consists of only skew-adjoint traceless endomorphisms of $\mathbb{C}^2$. Since $\mathfrak{u}(1)\otimes_{\mathbb{R}} \mathfrak{su}(2) \cong i\mathfrak{su}(2)$, $\mathfrak{u}(1)\otimes_{\mathbb{R}} \mathfrak{su}  (2)$ contains only self-adjoint traceless endomorphisms of $\mathbb{C}^2$. Then, it is clear that $\mu(\psi)$ can be viewed as an element of $\mathfrak{u}(1)\otimes_{\mathbb{R}} \mathfrak{su}(2)$. The fact that $\mu$ is a hyperk\"ahler moment map can be checked by direct computations, appealing to Definition \ref{hyperkahlermoment}.}

Although $0$ is a critical point of $\mu$, the set \textcolor{black}{$\mu^{-1}(0)\backslash\{0\}$} consist of only regular points. Hence one has the hyperk\"ahler reduction 
\begin{equation}\label{eq:framedchargedoneinstantons}
    \Big(\textup{Hom}_{\mathbb{C}}(E,\mathbb{H})\backslash\{0\}\Big)\ssslash U(1):=\Big(\mu^{-1}(0)\backslash\{0\}\Big)/U(1).
\end{equation}
A short computation with a good cup of Joe will reveal that $\mu^{-1}\lp0\rp\setminus\{0\}$ is diffeomorphic to $(0,\infty)\times \mathbb{S}\mathcal{F}$, where $\mathcal{F}\rightarrow\mathbb{S}E$ is the contact distribution of the unit sphere $\mathbb{S}E$ and $\mathbb{S}\mathcal{F}$ is the unit sphere bundle of the contact distribution. This implies that the hyperk\"ahler reduction $\mu^{-1}\lp0\rp\setminus\{0\}/U(1)$ is diffeomorphic to $(0,\infty)\times \mathbb{S}T\mathbb{P}E$, where $\mathbb{P}E$ is the projective space of $E$. It also follows from the ADHM construction \cite{atiyah1994construction,donaldson1997geometry} that
the hyperk\"ahler quotient \eqref{eq:framedchargedoneinstantons} 
can be identified with the  moduli space $\mathcal{M}_{\textup{fr}}(1,n)$ of framed charge one $\SU(n)$-instantons on $\mathbb{R}^4$ which is a  $4(n-1)$-dimensional  non-compact \textcolor{black}{hyperk\"ahler}  manifold. The action of $\SPN(3)$ on $\mathbb{H}$ and the action of $\SU(n)$ on $E$ induce an action of the Lie group $\SU(n)\times \SPN(3)$ on $\mathcal{M}_{\textup{fr}}(1,n)$. 
\end{Ex}

\subsection{Hyperk\"ahler bundles} The hyperk\"ahler structure gives a parallel action of $\text{Im}\mathbb{H}$ on the tangent bundle $M$:
\[\color{black}\gamma:M\times\text{Im}\mathbb{H}\to\textup{End}(TM)\]
\[(x,ai+bj+ck)\mapsto aI_x+bJ_x+cK_x\]
which satisfies the Clifford relations: 
\[\gamma(v)^2+|v|^2=0.\]
(Note that $\gamma$ can be equivalently given by an isometry between the unit sphere $\textup{S}(\text{Im}\mathbb{H})$ and \textcolor{black}{the fibers of $\mathfrak{b}(M)$, the bundle of 2-spheres of complex structures on the tangent spaces.}) Hence, one obtains the Clifford action
\begin{equation}\label{eq:cliffactionhyperkahler}
    \gamma:M\times \Cl(3)\to\textup{End}(TM)
\end{equation}
that is parallel with respect to the Riemannian metric.
Conversely, any parallel Clifford action \eqref{eq:cliffactionhyperkahler} is given by an oriented isometry 
\[\gamma:M\times\textup{S}(\text{Im}\mathbb{H})\to \mathfrak{b}(M).\]

This observation extends to hyperk\"ahler bundles over 3-manifolds due to Taubes \cite{taubes1999nonlinear}. Fix an oriented Riemannian 3-manifold $(Y,g_Y)$ which is always parallelizable. We then can write
\[\color{black}TY\simeq T^*Y\simeq P_{\SO}\times_{\SO(3)}\text{Im}\mathbb{H}.\]
Here $P_{\SO}$ is the bundle of oriented orthonormal frames of $Y$.

\begin{Def}\label{def:hyperkahlerbundlecliff}
    By a \emph{hyperk\"ahler bundle} over $Y$ we mean a fiber bundle $\pi:\mathscr{M}\to Y$ with fiber a hyperk\"ahler manifold $(M,g,I,J,K)$ with an isometry of bundles \textcolor{black}{$\gamma:\textup{S}(T^*Y)\buildrel\simeq\over\to\mathfrak{b}(\mathscr{M})$}. Here \textcolor{black}{$\textup{S}(T^*Y)$} is the unit tangent bundle and $\mathfrak{b}(\mathscr{M})$ is defined fiberwise as in \eqref{eq:complexstructuresphere}. By the discussion above, we then obtain an action
    \[\color{black}\gamma:\pi^*T^*Y\to \textup{End}(\mathcal{V}\mathscr{M})\]
    that extends to a bundle map {$\gamma:\pi^*\Cl(T^*Y)\to \textup{End}(\mathcal{V}\mathscr{M})$}, \textcolor{black}{where $\Cl(T^*Y)$ is the bundle of Clifford algebras associated with the Riemannian structure.}  Equivalently, one obtains a section of the tensor bundle
    \begin{equation}\label{eq:cliffordsection}
        \gamma\in\Gamma(\pi^*TY\otimes \textup{Hom}(\mathcal{V}\mathscr{M}))
    \end{equation}
    \textcolor{black}{which is parallel along the fibers of $\mathscr{M}$.}
\end{Def}

In this paper, the main examples of hyperk\"ahler bundles are associated bundles. For some compact Lie group $\mathscr{K}$, a principal $\mathscr{K}$-bundle $P$, and a hyperk\"ahler manifold $M$ with a $\mathscr{K}$-action, we form the hyperk\"ahler bundle $\mathscr{M}:=P\times_{\mathscr{K}} M$. A special case treated here is for $$M=\textup{Hom}_{\mathbb{C}}(\mathbb{C}^3,\mathbb{H})\backslash\{0\} \  \textup{or} \ \Big(\textup{Hom}_{\mathbb{C}}(\mathbb{C}^3,\mathbb{H})\backslash\{0\}\Big)\ssslash U(1),$$ $$\mathscr{K}=\SPNC(3),$$ and $P$ is a principal $\SPNC(3)$-bundle.


\subsection{The Fueter operator}
Let $\mathscr{M}\to Y$ be a hyperk\"ahler bundle and fix \textcolor{black}{an Ehresmann connection on it.} For a section $\phi\in\Gamma(\mathscr{M})$, we can pull back the Clifford section \eqref{eq:cliffordsection} to obtain:
\begin{equation}\label{eq:pullbackcliffordsection}
    \color{black}\phi^*\gamma\in\Gamma(Y,TY\otimes\phi^*\textup{End}(\mathcal{V}\mathscr{M})).
\end{equation}
Then we may compose this pullback section with the covariant derivative $\nabla\phi\in\Omega^1(Y,\phi^*\mathcal{V}\mathscr{M})$ to obtain the \emph{Fueter operator} $\mathfrak{F}$:
\begin{equation}\label{eq:fueteroperator}
   \mathfrak{F}\phi:=\phi^*\gamma \circ \nabla\phi\in\Gamma(Y,\phi^*\mathcal{V}\mathscr{M}). 
\end{equation}

    The Fueter operator is a nonlinear generalization of Dirac operators first introduced in \cite{taubes1999nonlinear}. When $N$ is a representation of the Clifford algebra $\Cl(3)$, one obtains a Clifford module $\mathscr{N}:=P_{\SPNC}\times_{\SPNC(3)}N$. Since $\mathscr{N}$ is a vector bundle, for any section $\phi\in \Gamma(\mathscr{N})$, one has the natural identification $\phi^*\mathcal{V}\mathscr{N}\simeq \mathscr{N}$, and hence the Fueter operator is the usual Dirac operator
    \[\mathfrak{F}\phi\in \Gamma (\mathscr{N}).\]

We can say a bit about the linearization of the Fueter operator. First observe that the Fueter operator defines a vector field on the infinite-dimensional manifold of global sections $\Gamma\lp Y,\mathscr{M}\rp$, where the tangent space at a section $\phi\in\Gamma\lp Y,\mathscr{M}\rp$ is given by $T_{\phi}\Gamma\lp Y,\mathscr{M}\rp\simeq \Gamma\lp Y,\phi^*\mathcal{V}\mathscr{M}\rp$. Thus, if we want to linearize the Fueter operator, we will need some notion of connection for the tangent bundle $T\Gamma\lp Y,\mathscr{M}\rp$. For this, we use the fact that $\mathscr{M}\rightarrow Y$ is a bundle of Riemannian manifolds which implies that the fibers have a Levi-Civita connection. In this case, we can define the parallel transport along a smooth one-parameter family of sections $\phi_t:[a,b]\rightarrow \Gamma\lp Y,\mathscr{M}\rp$ as follows: for $V\in\Gamma\lp\phi_t^*\mathcal{V}\mathscr{M}\rp$ and $y\in Y$, and $s\in[a,b]$ we let 
\begin{align*}
    P^{T\Gamma\lp Y,\mathscr{M}\rp}_{t\rightarrow s}:\Gamma\lp\phi^*_t\mathcal{V}\mathscr{M}\rp&\rightarrow \Gamma\lp Y,\phi^*_s\mathcal{V}\mathscr{M}\rp\\
    V&\mapsto \lp y\rightarrow P^{\mathscr{M}_y,LC}_{\phi_t(y)\rightarrow \phi_s(y)} V_y\rp
\end{align*}
Now observe that the pullback of the vertical bundle gives a \emph{Clifford module} 
\begin{equation}\label{eq:linearizedhyperkahler}
    \phi^*\mathcal{V}\mathscr{M}\to  Y.
\end{equation}
To see why this bundle is a Clifford module, first note that the action of quaternions on $\mathcal{V}\mathbb{M}$ can be interpreted as an action of the tangent bundle $TY\simeq Y\times \text{Im}\mathbb{H}$. Hence, one obtains a linear Dirac operator $D^{\phi}$ on the Clifford module \eqref{eq:linearizedhyperkahler}.

 \textcolor{black}{The following proposition is well-known. However, we could not pinpoint an exact reference to a proof of it in the literature. Below, we give a detailed proof of the statement.}

\begin{Prop}\label{p:linearizedfueter}
    \textcolor{black}{The linearization $\boldsymbol{\nabla}\mathfrak{F}_{\phi}$ of the Fueter operator \eqref{eq:fueteroperator} at $\phi$ is given by $D^{\phi}$. In particular, the linearization is elliptic.} 
\end{Prop}

\begin{proof}
We first describe the induced connection on the Clifford module $\phi^*\mathcal{V}\mathscr{M}$. We may identify the tangent space $T_{\phi}\Gamma(Y,\mathscr{M})$ with $\Gamma(Y,\phi^*\mathcal{V}\mathscr{M})$. Hence the linearization of the connection $\nabla$ at $\phi$ gives a map
\[\nabla^{\phi}:=L_{\phi}\nabla:\Gamma(Y,\phi^*\mathcal{V}\mathscr{M})\to \Gamma(Y,TY^*\otimes \phi^*\mathcal{V}\mathscr{M})\]
\[\psi\mapsto \frac{d}{dt}|_{t=0}P_{\phi_t}\nabla \phi_t\]
where $\phi_t\in \Gamma(T,\mathscr{M})$ is a curve of sections \textcolor{black}{with $\phi_0=\phi$} such that $\frac{d}{dt}|_{t=0}\phi_t=\psi$ and $P_{\phi_t}:\phi_t^{*}\mathcal{V}\mathscr{M}\to \phi_0^*\mathcal{V}\mathscr{M}$ is the bundle isomorphism given, for every $y\in Y$, by the parallel transport along the curve $s\mapsto \phi_s(y)$ in $\mathscr{M}_y$ via the Levi-Civita connection on the fiber $\mathscr{M}_y$.

The linearization $\nabla^{\phi}$ gives a connection on the vector bundle $\phi^*\mathcal{V}\mathscr{M}\to Y$. To see why this is a connection, take $\psi\in\Gamma(Y,\phi^*\mathcal{V}\mathscr{M})$ which can be represented by a curve of sections $\phi_t\in\Gamma(Y,\mathscr{M})$, i.e., \textcolor{black}{$\frac{\partial}{\partial t}\phi_t|_{t=0}=\psi$.} We need to show $\nabla^{\phi}f\cdot\psi=f\nabla^{\phi}\psi+df\cdot \psi$ for $f\in C^{\infty}(Y)$. First, note the ``tangent vector'' $f\cdot\psi$ corresponds to the curve $\phi_{f\cdot t}$ \textcolor{black}{given by $t\mapsto\phi_{f(y)t}(y)$ for each fixed $y\in Y$}. So, we may write
\[\nabla^{\phi}f\cdot \psi=\frac{d}{dt}|_{t=0}P_{\phi_{f.t}}\nabla \phi_{f\cdot t}=\frac{d}{dt}|_{t=0}(P_{\phi_{f.t}}{\pi^v d(\phi_{f\cdot t}}))\]
\[=\frac{d}{dt}|_{t=0}(P_{\phi_{f.t}}\pi^v(d\phi)|_{f\cdot t}+t\cdot df\cdot P_{\phi_{f.t}}\pi^v\frac{\partial}{\partial s}|_{s=f\cdot t}\phi_s)=f\nabla^{\phi}\psi+df\cdot\psi.\]
    Here $\pi^v:T\mathscr{M}\to\mathcal{V}\mathscr{M}$ is the projection along the horizontal bundle associated with the connection onto the vertical bundle, \textcolor{black}{$\nabla=\pi^v\circ d$, and  $d\phi|_{f.t}$ is given by $t\mapsto d\phi_s(y)|_{s=f.t}$ for each fixed $y\in Y$. The equality from the first to the second line follows from the equality $d(\phi_{f.t})=d\phi|_{f.t}+t.df.\frac{\partial}{\partial s}|_{s=f.t}\phi_s$.}

\textcolor{black}{To see why $L_{\phi}(\mathfrak{F})=D^{\phi}$, first note that the Dirac operator $D^{\phi}$ on $\phi^*\mathcal{V}\mathscr{M}$ is defined by
\[D^{\phi}=\phi^*\gamma\circ \nabla^{\phi}:\Gamma(Y,\phi^*\mathcal{V}\mathscr{M})\to \Gamma(Y,\phi^*\mathcal{V}\mathscr{M}).\]}
\textcolor{black}{Since $\gamma$ is parallel along fibers of $\mathscr{M}$, we have $P_{\phi_t}\phi_t^*\gamma P_{\phi_t}^{-1}=\phi_0^*\gamma$. So, we may write
\[L_{\phi}(\mathfrak{F})\psi=\frac{d}{dt}|_{t=0}P_{\phi_t}\mathfrak{F}\phi_t=\frac{d}{dt}|_{t=0}P_{\phi_t}\phi_t^*\gamma\circ \nabla\phi_t=\frac{d}{dt}|_{t=0}\big(P_{\phi_t}\phi_t^*\gamma P_{\phi_t}^{-1} \circ P_{\phi_t}\nabla\phi_t\big)\]
\[=\frac{d}{dt}|_{t=0}\big(\phi_0^*\gamma \circ P_{\phi_t}\nabla\phi_t\big)=\phi^*\gamma\circ L_{\phi}\nabla\psi=D^{\phi}\psi,\]}
as claimed.
\end{proof}

\begin{Rem}\label{r:symbolofDirac}
    The connection $\nabla^{\phi}$ on $\phi^*\mathcal{V}\mathscr{M}\to Y$ does not need to be a Clifford connection \cite{berline2003heat}, so the Dirac operator $D^{\phi}$ does not share many of the geometric properties of the classical Dirac operators (such as the Lichnerowicz-Weitzenb\"ock formula, etc.). However, its symbol is still given by the Clifford multiplication by the corresponding covectors.
\end{Rem}

\section{An aquaternionic moduli space}\label{s:aquaternionicmodulispace} 

In this section, we fix the following notations:
\[{M}:=\textup{Hom}_{\mathbb{C}}(\mathbb{C}\otimes_{\mathbb{R}} \text{Im}\mathbb{H},\mathbb{H})\backslash\{0\},\]
\textcolor{black}{which is hyperk\"ahler manifold with a $U(1)$-equivariant moment map $\mu:M\to i\mathfrak{su}(2)$ for which $0$ is a regular value (cf. Example \ref{ex:mainhyperkahlerquotient}). We then denote by $M^{\mu}$ the inverse image $\mu^{-1}(0)$ and the hyperk\"ahler quotient by }
\begin{equation}{\label{eq:hyperkahlerquotient}
\color{black}{M}_0:=M^{\mu}/U(1).  }  
\end{equation}
We also denote $M^+:=\textup{Hom}_{\mathbb{C}}(\mathbb{C}\otimes_{\mathbb{R}} \text{Im}\mathbb{H},\mathbb{H})$, which is also a hyperk\"ahler manifold with $U(1)$-equivariant moment map $\mu:M^+\rightarrow i\mathfrak{su}\lp 2\rp$, but for which $0$ is no longer a regular value.

\subsection{$\textup{Spin}\lp 3\rp$ action}
The \textcolor{black}{set-up in \eqref{eq:hyperkahlerquotient}} was crucial in the description of the noncompactness theorem for the multi-spinor Seiberg-Witten equation \cite{haydys2015compactness} in terms of nonlinear harmonic spinors. \textcolor{black}{We have an action of  $\textup{Spin}(3)$ on $\mathbb{C}\otimes_{\mathbb{R}} \text{Im}\mathbb{H}$ where $\textup{Spin}(3)$ acts on the $\text{Im}\mathbb{H}$-factor.} Also, $\textup{Spin}(3)$ acts on $\mathbb{H}$ by restricting the action of the Clifford algebra $\Cl(3)$ to the $\textup{Spin}(3)$ group. Therefore, we have an action of $ \textup{Spin}(3)$ on $\textup{Hom}_{\mathbb{C}}(\mathbb{C}\otimes_{\mathbb{R}} \text{Im}\mathbb{H},\mathbb{H})$ with respect to which the submanifolds $M$ and $M^{\mu}=\mu^{-1}(0)\backslash\{0\}$ are invariant. So the hyperk\"ahler quotient $M_0=M^{\mu}/U(1)$ carries an action of \textcolor{black}{ $\textup{Spin}\lp 3\rp$.}


\subsection{An aquaternionic quotient}

\textcolor{black}{As in Section \ref{ss:3/2-spionors}, we have an orthogonal decomposition
\[\textup{Hom}_{\mathbb{C}}(\mathbb{C}\otimes \text{Im}\mathbb{H},\mathbb{H})=\iota(\mathbb{H})\oplus W\]
where $\iota:\mathbb{H}\to \textup{Hom}_{\mathbb{C}}(\mathbb{C}\otimes_{\mathbb{R}} \text{Im}\mathbb{H},\mathbb{H})$ is an $SU(2)$-equivariant embedding and $W$ is the vector space of $3/2$-spinors given as the kernel of the Clifford action. This decomposition is}
preserved by the diagonal action of $\text{Spin}\lp 3\rp$. Denote the image of $W^{\mu}:=W\cap\Big(\mu^{-1}(0)\backslash\{0\}\Big)$ under the projection $M\to M_0$ by $W_0$ which carries an action of $\text{Spin}(3)$. Note that if we add in right multiplication by elements in $U(1)$, then the action by $\text{Spin}\lp 3\rp$ is enlarged to action of $\text{Spin}^c\lp 3\rp\simeq U\lp 2\rp$. We call $W_0$ the \textcolor{black}{\emph{aquaternionic} moduli space of \textcolor{black}{$\SU(3)$-instantons}, and we will show this is a 4-dimensional manifold. The reason for calling this space `aquaternionic' is that it is formed by applying the hyperk\"ahler quotient to the non-hyperk\"ahler submanifold $W$, which, as a result, $W_0$ does not carry a hyperk\"ahler structure.}\par
The proposition below provides a local description of $W_0$, but in the future, we would like to understand more about the global properties of this moduli space. In general, the subspace $W^{\mu}$ of $W$ will be a quadratic variety and certainly has a conical singularity at the origin. Our concern is whether $W^{\mu}\setminus\{0\}$ is \emph{smooth}. Being a quadratic variety, a priori there could be other singularities within $W^{\mu}$ away from the origin. The key problem in checking smoothness is that the map $\mu$ does not interact cleanly with the defining condition of $W$, which leads to difficulty when checking that $\mu$ has rank $3$ along $W^{\mu}$. We use the background $U(2)$ symmetry to ascertain smoothness.
\begin{Prop}\label{p:w_0ismanifold}
    The moduli space $W_0$ is a 4-dimensional submanifold of $M_0$.
\end{Prop}
\begin{proof}
If we can prove that the quadratic variety $W^{\mu}\setminus\{0\}$ is smooth, then because the action of $U\lp1\rp$ on $W^{\mu}\setminus\{0\}$ is free, this would imply the smoothness of $W_0$.
Firstly, note that as a complex representation of $\text{Spin}\lp3\rp\simeq\text{SU}\lp 2\rp$, the action on $\text{Hom}_{\mathbb{C}}\lp\mathbb{C}\otimes_{\mathbb{R}}\text{Im}\mathbb{H},\mathbb{H}\rp$ is equivalent to $\mathfrak{sl}\lp 2,\mathbb{C}\rp\otimes_{\mathbb{C}}\mathbb{C}^2$ with the adjoint action on the $\mathfrak{sl}\lp 2,\mathbb{C}\rp$ component and the standard action on $\mathbb{C}^2$. If 
\[
    X=\begin{bmatrix}
        0 & 1 \\
        0 & 0 
    \end{bmatrix},\quad Y=\begin{bmatrix}
        0 & 0 \\ 
        1 & 0
    \end{bmatrix},\quad H=\begin{bmatrix}
        1 & 0 \\
        0 & -1
    \end{bmatrix},\quad \widehat{H}=\frac{1}{\sqrt{2}}H
\]
is the standard basis for $\mathfrak{sl}\lp 2,\mathbb{C}\rp$, then for $\vec{z},\vec{w}\in\mathbb{C}^3$ one has
\begin{align*}
    \mu\lp\lp z_1X+z_2Y+z_3\widehat{H}\rp\otimes\vec{e}_1+\lp w_1X+w_2Y+w_3\widehat{H}\rp\otimes\vec{e}_2\rp=\begin{bmatrix}
        \frac{|\vec{z}|^2-|\vec{w}|^2}{2} & \vec{w}^*\vec{z} \\
        \vec{z}^*\vec{w} & \frac{|\vec{w}|^2-|\vec{z}|^2}{2}
    \end{bmatrix}.
\end{align*}
We also obtain that
\[
    W=\mathbb{C}\cdot\left\{Y\otimes\vec{e}_2,X\otimes\vec{e}_1,Y\otimes\vec{e}_1+H\otimes\vec{e}_2,H\otimes\vec{e}_1-X\otimes\vec{e}_2\right\}
\]
which implies that a generic element of $W$ will take the form
\[
    aY\otimes\vec{e}_2+bX\otimes\vec{e}_1+c\lp Y\otimes\vec{e}_1+H\otimes\vec{e}_2\rp+d\lp H\otimes\vec{e}_1-X\otimes\vec{e}_2\rp,\quad a,b,c,d\in\mathbb{C}
\]
which we can re-express the above as
\[
    \lp \sqrt{2}d\widehat{H}+bX+cY\rp\otimes\vec{e}_1+\lp \sqrt{2}c\widehat{H}-dX+aY\rp\otimes\vec{e}_2.
\]
The above element will lie in $\mu^{-1}\lp0\rp$ when the system of equations is satisfied
\[
    2|d|^2+|b|^2+|c|^2=2|c|^2+|d|^2+|a|^2,\quad 2\overline{d}c-\overline{b}d+\overline{c}a=0
\]
which, by rearranging, is equivalent to the system
\[
    |d|^2+|b|^2=|c|^2+|a|^2,\quad 2\overline{d}c-\overline{b}d+\overline{c}a=0,\quad \lp\star\rp
\]
We will analyze the above system, but first a word from our sponsor: group theory!\par
As we have discussed above, the space $W$ is naturally acted on by $\textup{SU}\lp 2\rp$ and has weight $m=3/2$. We can enlarge this action to an action by $U(2)$ by adding in the complex multiplication action on $W$ by $U(1)$. However, there is \emph{alternative action} of $U(2)$ on $W$ given as follows: As an $\text{SU}(2)$ representation, $W$ is isomorphic to the action of $\text{SU}(2)$ on $\text{Sym}^3\lp\mathbb{C}^2\rp$ by pullback and this action naturally lifts to an action of $U(2)$ by pullback. However the action given above by $U(2)$ and the action by pullback are \emph{not} the same action. For the covering group homomorphism
\begin{align*}
    D_3:SU(2)\times U(1)&\rightarrow SU(2)\times U(1)\\
    \lp U,z\rp&\mapsto \lp U,z^3\rp
\end{align*}
this descends to a degree three covering group homomorphism
\[
    D_3:U(2)\rightarrow U(2).
\]
If we let $U(2)\circlearrowright_PW$ denote the action by pullback, and $U(2)\circlearrowright_CW$ denote the canonical action given above, we have that
\[
    PD_3=C.
\]
The importance of the above statement is the following: for computing the size of $U(2)$-orbits within $W^{\mu}=\mu^{-1}(0)\cap W$, we are free to interchange between the $P$ action and $C$ action. 
The $P$-action is simplest to compute with, thus suppose that some element $w\in \text{Sym}^3\lp\mathbb{C}^2\rp$ has an isotropy group with dimension $\geq 1$. Then, after conjugation, there is some $\theta_1,\theta_2\in\mathbb{R}$, $\theta_1\neq\theta_2$, for which
\[
    \begin{bmatrix}
        e^{i\theta_1 t} & 0 \\
        0 & e^{i\theta_2t}
    \end{bmatrix}\cdot w=w,\quad\text{for all $t\in\mathbb{R}$}
\]
where the $\cdot$ denotes the usual action of $U(2)$ on $\text{Sym}^3\lp\mathbb{C}^2\rp$ by pullback. If we express 
\[
    w=az_1^3+bz_1^2z_2+cz_1z_2^2+dz_2^3
\]
then a case-by-case analysis will reveal that if $w\neq0$ is left fixed by the action of the above element of $U(2)$, then all but one of the coefficients above $a,b,c,d$ must vanish. But this means that up to a constant $w$ (and up to the action of some element $A\in U(2)$), there are only four possibilities for $w$, namely the canonical basis elements $z_1^3$, $z_1^2z_2$, $z_1z_2^2$, or $z_2^3$. However, up to constants, the basis elements $z_1^3$, $z_1^2z_2$, $z_1z_2^2$, or $z_2^3$ correspond to the matrices tensored with vectors above. We are now interested in whether non-zero points with positive-dimensional isotropy lie within $W^{\mu,\times}$. If $w\in W^{\mu,\times}$, then so is $A\cdot w$ for any $A\in U(2)$. Thus, if $w$ has positive-dimensional isotropy, so does $A\cdot w$ for any $A\in U(2)$ which implies that we can change coordinates if necessary and assume that $w$ is a non-zero multiple of one of the matrix-vector tensor products above. We rewrite each of the above matrix-vector tensor products using the usual Pauli-spin matrix basis for $\mathfrak{su}(2)$
\[
    E_1=\begin{bmatrix}
        i & 0 \\
        0 & -i
    \end{bmatrix},\quad E_2=\begin{bmatrix}
        0 & -1 \\
        1 & 0
    \end{bmatrix},\quad E_3=\begin{bmatrix}
        0 & i \\
        i & 0
    \end{bmatrix}
\]
giving us the four possibilities (up to scale) for $w$
\begin{align*}
    &\frac{1}{2}\lp E_2-iE_3\rp\otimes_{\mathbb{C}}\begin{bmatrix}
        0 \\ 1
    \end{bmatrix},\quad \frac{1}{2}\lp E_2-iE_3\rp\otimes_{\mathbb{C}}\begin{bmatrix}
        1 \\ 0
    \end{bmatrix}-iE_1\otimes_{\mathbb{C}}\begin{bmatrix}
        0 \\ 1
    \end{bmatrix}\\
    &\frac{1}{2}\lp iE_3+E_2\rp\otimes_{\mathbb{C}}\begin{bmatrix}
        0 \\ 1
    \end{bmatrix}-iE_1\otimes_{\mathbb{C}}\begin{bmatrix}
        1 \\ 0
    \end{bmatrix},\quad -\frac{1}{2}\lp iE_3+E_2\rp\otimes_{\mathbb{C}}\begin{bmatrix}
        1 \\ 0
    \end{bmatrix}
\end{align*}
and none of the above elements lie within $W^{\mu,\times}$. Thus, all $Spin^c\lp 3\rp$-orbits within $\mathbb{S}\lp W_y\rp\cap W^{\mu,\times}$ are four dimensional. Now let us check that $\mathbb{S}W\cap W^{\mu,\times}$ is a \emph{finite} collection of $U(2)$-orbits.\par
The $W^{\mu}$ equations $\lp\star\rp$ are $\mathbb{R}_{>0}$-scale invariant and also $U(1)$-invariant. Thus, we can fix a scale by setting 
\[
    2|d|^2+|b|^2+|c|^2=2|c|^2+|d|^2+|a|^2=\frac{1}{2}.
\]
The reason for the above choice of scale is to ensure that the vector 
\[
    aY\otimes\vec{e}_2+bX\otimes\vec{e}_1+c\lp Y\otimes\vec{e}_1+H\otimes\vec{e}_2\rp+d\lp H\otimes\vec{e}_1-X\otimes\vec{e}_2\rp
\]
lies within the unit sphere of $W$, $\mathbb{S}W$. If we can show that $W^{\mu}\cap\mathbb{S}W$ is $4$-dimensional, it must be a finite union of $U(2)$ orbits and we will have achieved smoothness. \par
First, note that the first equations in $\lp\star\rp$ along with the constraint of lying on $\mathbb{S}W$ imply that
\[
    |d|^2+|b|^2=|c|^2+|a|^2,\quad 0\leq |d|,|c|\leq \frac{1}{2}.
\]
We now use the $U(1)$-invariance of the equations to rotate $d$ to the non-negative real axis $[0,\infty)$. Thus we may assume that $d\geq 0$ and now turn to the second equation $\lp\star\rp$ which now reads
\[
    2dc+\overline{c}a=d\overline{b},\quad \lp\star_2\rp.
\]
If $d=0$, then we have that $\overline{c}a=0$ which implies that either $c=0$ and $|b|^2,|a|^2=\frac{1}{2}$ with no further constraints on $a$ or $b$ (thus we obtain a $2$-torus in the complex $\lp a,b\rp$-plane), or $a=0$ and we see that $|c|^2=\frac{1}{4}$ which implies that $|b|^2=\frac{1}{4}$ and no further constraint on $b$ (thus we obtain a circle in the $\lp a,b\rp$-plane). Thus let us now assume that $0<d\leq \frac{1}{2}$. Then we can divide by $d$ in $\lp\star_2\rp$ and obtain
\[
    \overline{b}=2c+\frac{\overline{c}}{d}a,\quad \lp\star_3\rp.
\]
Note that for $0\leq|c|\leq\frac{1}{2}$ and $0<d\leq\frac{1}{2}$ fixed, we are searching for all $\lp a,b\rp$ satisfying 
\[
    |b|^2=\frac{1}{2}-2d^2+|c|^2,\quad |a|^2=\frac{1}{2}-2|c|^2-d^2,\quad \overline{b}=2c+\frac{\overline{c}}{d}a.
\]
However, letting $a$ vary on the circle with radius squared $\frac{1}{2}-2|c|^2-d^2$, the set of points $2c+\frac{\overline{c}}{d}a$ is a circle in the complex plane centered at $2c$. Letting $b$ vary over all points satisfying $|b|^2=\frac{1}{2}-2d^2-|c|^2$, we also obtain a circle in the complex plane centered at $0$. Thus we are interested in counting the number of intersection points between two circles in the complex plane. Observe that when $c=0$, we are forced to put $b=0$, $d=\frac{1}{2}$ and $a$ can take any value satisfying $|a|^2=\frac{1}{4}$ (thus we obtain a circle in the $\lp a,b\rp$-plane). So we will assume that $c\neq0$ and thus these two circles have distinct centers which implies their size of intersection is bounded above by $2$. All of this to say the following: If we let $ \Pi_{\lp c,d\rp}:W^{\mu}\cap\mathbb{S}_1W\rightarrow \mathbb{C}^2_{\lp c,d\rp}$ be the projection to the $\lp c,d\rp$ plane, we see that outside a codimension $3$ subset, the projection is at most $2$-to-$1$ and along the codimension $3$ subset the preimage is either a circle or a $2$-torus. Thus $W^{\mu}\cap\mathbb{S}_1W$ has topological dimension at most $4$. Because $U(2)$-orbits are closed, connected, $4$-dimensional smooth manifolds, we see that $W^{\mu}\cap\mathbb{S}_1W$ must be a full $U(2)$-orbit, hence smooth.\par
\end{proof}
\begin{figure}[htbp]
    \centering
    \includegraphics[width=0.7\textwidth]{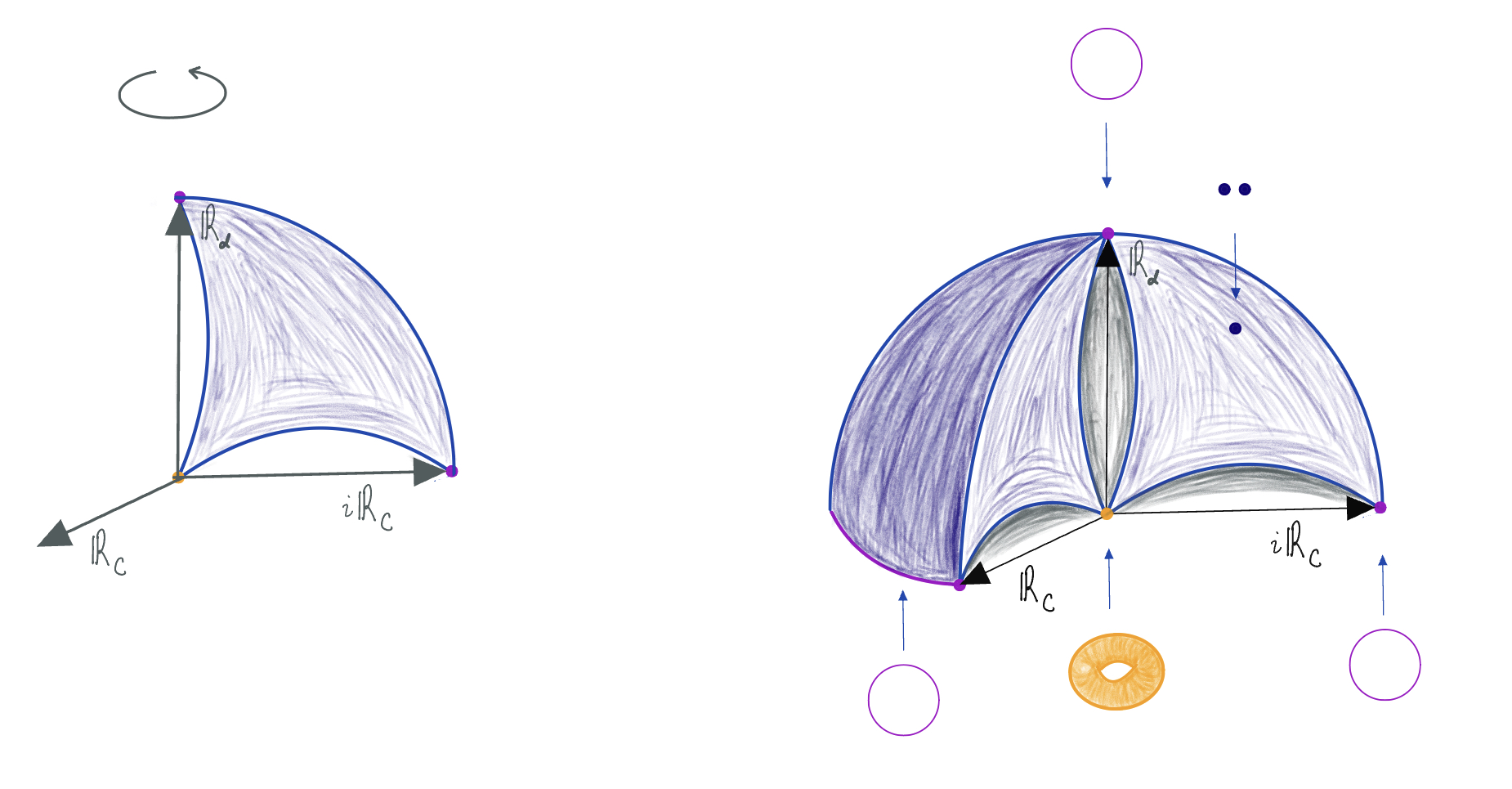}
    \caption{On the left diagram above, the purple-shaded region represents the projection of $W^\mu \cap \mathbb S_1 W$ into the real half-space spanned by $\mathbb R_d \times i\mathbb R_c$. On the right diagram above, we have the rotation of the left diagram within the upper-half space $\mathbb R_d \times \mathbb{C}_c$. The 2 points, circles and tori represent the preimage of points under the projection.}
    \label{fig:my_figure} 
\end{figure}

\subsection{A normal bundle on $W^{\mu}$} \textcolor{black}{Recall} the orthogonal decomposition 
\[\textup{Hom}_{\mathbb{C}}(\mathbb{C}\otimes_{\mathbb{R}} \text{Im}\mathbb{H},\mathbb{H})=\iota(\mathbb{H})\oplus W\]
where $\iota:\mathbb{H}\to \textup{Hom}_{\mathbb{C}}(\mathbb{C}\otimes_{\mathbb{R}} \text{Im}\mathbb{H},\mathbb{H})$ is the embedding \eqref{eq:iotamap}. As $M$ is an open subset of $\textup{Hom}_{\mathbb{C}}(\mathbb{C}\otimes_{\mathbb{R}} \text{Im}\mathbb{H},\mathbb{H})$, $TM|_W$ is a trivial bundle and using the orthogonal decomposition above, one has a trivial subbundle whose fibers are identified canonically with $\iota(\mathbb{H})$ that is orthogonal to $TW$. With abuse of notation, we denote this trivial subbundle with $\iota(\mathbb{H})$ again. Consequently,   $\iota(\mathbb{H})$ is orthogonal to the fibers of $TW^{\mu}$. In particular, $\iota(\mathbb{H})$ is orthogonal to the fibers of $\mathcal{L}|_{W^{\mu}}.$ To describe the intersection of $TM^{\mu}$ with $\iota\lp\mathbb{H}\rp$ it will be important for us to have a formula for the differential of $\mu$. Using the identification $\textup{Hom}_{\mathbb{C}}(\mathbb{C}\otimes_{\mathbb{R}} \text{Im}\mathbb{H},\mathbb{H})\simeq \mathfrak{sl}\lp 2,\mathbb{C}\rp\otimes_{\mathbb{C}}\mathbb{C}^2$, we have that at $\lp\vec{z},\vec{w}\rp=\lp\vec{x},\vec{y},\vec{u},\vec{v}\rp$, where $\vec{z}=\vec{x}+i\vec{y}$ and $\vec{w}=\vec{u}+i\vec{v}$, the formula
\[
    D\mu\lp\vec{x},\vec{y},\vec{u},\vec{v}\rp=\begin{bmatrix}
        \vec{x}\cdot & \vec{y}\cdot & -\vec{u}\cdot & -\vec{v}\cdot \\
        \vec{u}\cdot & \vec{v}\cdot & \vec{x}\cdot & \vec{y}\cdot \\
        -\vec{v}\cdot & \vec{u}\cdot & \vec{y}\cdot & -\vec{x}\cdot
    \end{bmatrix}
\]
where all the above rows denote dot products with the respective vectors and we have used the identification $\mathfrak{su}\lp 2\rp\simeq\mathbb{R}^3$ to obtain the formula
\[
    \mu\lp \vec{z},\vec{w}\rp=\lp\frac{|\vec{z}|^2-|\vec{w}|^2}{2},\text{Re}\la\vec{z},\vec{w}\ra_{\mathbb{C}^3},\text{Im}\la\vec{z},\vec{w}\ra_{\mathbb{C}^3}\rp.
\]
\begin{Prop}\label{p:specialnormalbundle}
    The intersection of the bundles $N:=TM^{\mu}|_{W^{\mu}}\cap \iota(\mathbb{H})$ is an $S^1$-invariant real rank $2$ vector bundle over $W^{\mu}$ orthogonal to the line bundle $\mathcal{L}|_{W^{\mu}}$. \textcolor{black}{Also, $N$ is invariant under $\spinc(3)$-action.}
\end{Prop}

\begin{proof}
    We let $\iota^T\lp\mathbb{H}\rp$ denote the translation-invariant distribution of $TM$ obtained by using the linear structure of $M^+$. Under the natural actions of $U(2)$ and $\mathbb{R}_{>0}$ on $M$, we see that $\iota^T\lp\mathbb{H}\rp$ is a $U\lp 2\rp\times\mathbb{R}_{>0}$ equivariant vector bundle on $M$. The tangent bundle $TM^{\mu}$ is also a $U\lp 2\rp\times \mathbb{R}_{>0}$-equivariant vector bundle. Because $W^{\mu}$ is a full $U\lp 2\rp\times\mathbb{R}_{>0}$-orbit, we see that the the intersection $TM^{\mu}\big|_{W^{\mu}}\cap\iota^T\lp\mathbb{H}\rp$ has constant rank and is a smooth vector bundle along $W^{\mu}$. To compute its rank, one only needs to compute at a well-chosen point on $W^{\mu}$.\par
    Using the identification $M^+\simeq\mathfrak{sl}\lp 2,\mathbb{C}\rp$ used within the proof above, we have that 
    \[
        W^{\perp}=\mathbb{C}\cdot\left\{ -2Y\otimes_{\mathbb{C}}\vec{e}_1+H\otimes_{\mathbb{C}}\vec{e}_2,H\otimes_{\mathbb{C}}\vec{e}_1+2X\otimes_{\mathbb{C}}\vec{e}_2\right\}.
    \]
    In this case note that the point $\lp \vec{z},\vec{w}\rp=\lp1/\sqrt{2},0,0,0,1/\sqrt{2},0\rp$ satisfies system $\lp\star\rp$ and thus lies in $W^{\mu}$. At this point we have that $D\mu\lp 1/\sqrt{2},0,0,0,1/\sqrt{2},0\rp$ is row equivalent to the matrix
    \[
        \begin{bmatrix}
            1 & 0 & 0 & 0 & 0 & 0 & 0 & -1 & 0 & 0 & 0 & 0 \\
            0 & 1 & 0 & 0 & 0 & 0 & 1 & 0 & 0 & 0 & 0 & 0 \\
            0 & 0 & 0 & 0 & 1 & 0 & 0 & 0 & 0 & -1 & 0 & 0
        \end{bmatrix}
    \]
    In which case for $a,b\in\mathbb{C}$, the element
    \begin{align*}
        a\lp -2Y\otimes_{\mathbb{C}}\vec{e}_1+H\otimes_{\mathbb{C}}\vec{e}_2\rp+&b\lp H\otimes_{\mathbb{C}}\vec{e}_1+2X\otimes_{\mathbb{C}}\vec{e}_2\rp=\\
        & \lp -2aY+\sqrt{2}b\widehat{H}\rp\otimes_{\mathbb{C}}\vec{e}_1+\lp 2bX+\sqrt{2}a\widehat{H}\rp\otimes_{\mathbb{C}}\vec{e}_2
    \end{align*}
    lies within the kernel of $D\mu\lp 1/\sqrt{2},0,0,0,1/\sqrt{2},0\rp$ if and only if $b=\overline{a}$. Thus we see that the rank of $N$ at the point $\lp 1/\sqrt{2},0,0,0,1/\sqrt{2},0\rp$ is $2$.
\end{proof}

\begin{Rem}\label{r:pushforwardbundle}
    Since $N$ is $S^1$-invariant \textcolor{black}{and} orthogonal to  $\mathcal{L}|_{W^{\mu}}$, we obtain a \emph{pushforward subbundle} $\tau_*N$  of $TM_0|_{W_0}$ which is $\SO\lp 3\rp$-equivariant.
\end{Rem}
If $\pi_{3/2}:M^+\rightarrow W$ denotes the orthogonal projection onto $W$, for a point $w\in W^{\mu}$ we will also need to know the dimension of $\pi_{3/2}T^{\perp}_wM^{\mu}\subseteq W$. For this computation, we define a bundle $H^{\parallel}\rightarrow W^{\mu}$ which gives the orthogonal splitting
\[
    T_wM^{\mu}=T_wW^{\mu}\oplus H^{\parallel}\oplus N.
\]
\begin{Prop}\label{dimension3/2tperp}
The dimension of $\pi_{3/2}T^{\perp}_wM^{\mu}$ is $3$.
\end{Prop}
\begin{proof}
    First observe that by $\SPNC\lp3\rp$ equivariance, the computation can be done at a well-chosen point; thus we continue with the computation we did above with the point $\lp1/\sqrt{2},0,0,0,1/\sqrt{2},0\rp$, and note that a basis for $N$ at $\lp1/\sqrt{2},0,0,0,1/\sqrt{2},0\rp$ is given by
    \[
        \left\{ -2e_1+\sqrt{2}e_3+2e_7+\sqrt{2}e_9,-2e_5-\sqrt{2}e_6-2e_{10}+\sqrt{2}e_{12}\right\}
    \]
    which implies that a basis for $N^{\perp}\leqslant \iota^L\lp\mathbb{H}\rp$ is given by
    \[
        \left\{ -2e_2-\sqrt{2}e_3-2e_7+\sqrt{2}e_9,-2e_5+\sqrt{2}e_6+2e_{10}+\sqrt{2}e_{12}\right\}.
    \]
    A basis for $\ker D\mu\lp 1/\sqrt{2},0,0,0,1/\sqrt{2},0\rp$ is 
    \[
        \left\{ e_3,e_4,e_6,e_2-e_7,e_1+e_8,e_9,e_5+e_{10},e_{11},e_{12}\right\}
    \]
    and note that any nontrivial linear combination of the computed basis elements of $N^{\perp}$ is \emph{not} orthogonal to all basis elements of $\ker D\mu\lp 1/\sqrt{2},0,0,0,1/\sqrt{2},0\rp$. This can be seen by using the relations
    \begin{align*}
        &e_3\cdot\lp -2e_2-\sqrt{2}e_3-2e_7+\sqrt{2}e_9\rp=-\sqrt{2},\quad e_3\cdot\lp -2e_5+\sqrt{2}e_6+2e_{10}+\sqrt{2}e_{12} \rp=0\\
        &e_6\cdot\lp -2e_2-\sqrt{2}e_3-2e_7+\sqrt{2}e_9 \rp=0,\quad e_6\cdot\lp -2e_5+\sqrt{2}e_6+2e_{10}+\sqrt{2}e_{12} \rp=\sqrt{2}.
    \end{align*}
    This implies that the rank of the intersection bundle $N^{\perp}\cap T^{\perp}M^{\mu}\big|_{W^{\mu}}$ is zero, which implies that when restricted to the bundle $T^{\perp}M^{\mu}$, the projection $\pi_{3/2}$ is injective.
\end{proof}


\subsection{Two examples of  hyperk\"ahler bundles}
    Let $M$, $M^+$, and $M_0=M\ssslash U(1)$ be as in Section \ref{s:aquaternionicmodulispace}. These are hyperk\"ahler manifolds which carry actions of $\SPNC(3)$. Consider an oriented Riemannian manifold $Y$ with its canonical spin structure. \textcolor{black}{Denote by $P_{\SO}$ and $P_{\SPNC}$ its associated principal $\SO(3)$- and $\SPNC(3)$-bundles and define the  associated bundles $\mathscr{M}:=P_{\SPNC\lp 3\rp}\times_{\SPNC(3)}M$, $\mathscr{M}^+:=P_{\SPNC\lp 3\rp}\times_{\SPNC(3)}M^+$, and $\mathscr{M}_0:=P_{\SPNC\lp 3\rp}\times_{\SPNC(3)}M_0$ are the hyperk\"ahler bundles. Note that $\mathscr{M}^+$ is just the twisted spinor bundle $T^*Y^{\mathbb{C}}\otimes_{\mathbb{C}}\slashed{S}$.}
\subsection{\textcolor{black}{Special subbundles of hyperk\"ahler bundles}}\label{ex:mainhyperkahlerbundles} Since $W\subset M^+$ and $W_0\subset M_0$ are invariant under the action of $\SPNC\lp 3\rp$, we obtain the associated subbundles
\begin{equation}\label{eq:simpleidentificationof3/2spinors}
        \mathscr{W}\simeq P_{\SPNC}\times_{\SPNC(3)}W\subset \mathscr{M}
    \end{equation}
    and
    \begin{equation}\label{eq:simpleidentificationofaquaternionsinstantonbundle}\color{black}
        \mathscr{W}_0\simeq P_{\SPNC}\times_{\SPNC(3)}W_0\subset\mathscr{M}_0.
    \end{equation}
    Note that $\mathscr{W}$ is the bundle of 3/2-spinors. The bundles $\mathscr{W}$ and $\mathscr{W}_0$ do not inherit hyperk\"ahler bundle structures.\par
    The bundles $\mathscr{M}_0$ and $\mathscr{W}_0$ admit a reduction of structure group when one notes that the homomorphism
    \[
        \rho:\SPNC\lp 3\rp\rightarrow \text{Diff}\lp M_0\rp
    \]
    descends to a homomorphism
    \[
        \widetilde{\rho}:\SO\lp 3\rp\rightarrow \text{Diff}\lp M_0\rp
    \]
    with the same for $W_0$. Thus, we only require the orthogonal frame bundle $P_{\SO}$ of $Y$ to construct the fibrations $\mathscr{M}_0$ and $\mathscr{W}_0$. What this implies is that if $\mathscr{W}$ is constructed from the $\SPNC\lp3\rp$-structure $P_{\SPNC\lp 3\rp}\rightarrow Y$, and $L\rightarrow Y$ is a complex hermitian line bundle, then the new bundle of $3/2$-spinors $\mathscr{W}^L:=\mathscr{W}\otimes_{\mathbb{C}}L$ gives rise to a new aquaternionic fibration $\mathscr{W}^L_0\rightarrow Y$ which is diffeomrophic as fibrations to the original aquaternionic fibration
    \[
        \mathscr{W}^L_0\simeq \mathscr{W}_0.
    \]
    \par
    By Proposition \ref{p:specialnormalbundle}, the normal bundle $N\rightarrow W^{\mu}\setminus\{0\}$ is a real rank $3$ $\SPNC\lp 3\rp$-equivariant vector bundle on $W^{\mu}\setminus\{0\}$. If we define $\mathscr{W}^{\mu}:=P_{\SPNC\lp 3\rp}\times_{\SPNC\lp 3\rp}W^{\mu}\setminus\{0\}$, then we obtain
    \[
        \mathcal{N}:=P_{\SPNC\lp 3\rp}\times_{\SPNC\lp 3\rp}N\rightarrow\mathscr{W}^{\mu}.
    \]

\subsection{The 3/2-Fueter operator}\sloppy
In Subsection \ref{ex:mainhyperkahlerbundles} we introduced the fiber bundle $\mathscr{W}_0=P_{\SO}\times_{\widetilde{\rho}}W_0$ , which is subbundle of a hyperk\"ahler bundle $\mathscr{M}_0=P_{\SO}\times_{\widetilde{\rho}}M_0$ that are associated bundles for the natural of action of $\SPNC\lp 3\rp$. 
We define a version of the Fueter operator for $\mathscr{W}_0$. 
Let $p:TM_0|_{W_0}\to TW_0$ be the orthogonal projection map, which induces a projection
\begin{equation}\label{eq:orthogonalprojectionofverticalbundles}
    p_0:\mathcal{V}\mathscr{M}_0|_{\mathscr{W}_0}\to \mathcal{V}\mathscr{W}_0.
\end{equation}
\begin{Def}\label{def:3/2-fueter}
    \textcolor{black}{Fix the Levi-Civita connection on $P_{\SO}$. Any section of the bundle $\mathscr{W}_0$ can be considered as a section of the bundle $\mathscr{M}_0$. Let then $\Psi\in\Gamma(\mathscr{W}_0)$ be a section, and apply the Fueter operator on $\mathscr{M}_0$ to obtain a section $\mathfrak{F}\Psi\in\Gamma(Y,\Psi^*\mathcal{V}\mathscr{M}_0)$. Applying the projection map \eqref{eq:orthogonalprojectionofverticalbundles}, we obtain the \emph{3/2-Fueter operator} $\mathfrak{Q}$ by}

    \[\color{black}\mathfrak{Q}\Psi:=p_0\circ\mathfrak{F}\Psi\in \Gamma(Y,\Psi^*\mathcal{V}\mathscr{W}_0).\]
    
\end{Def}


\section{A Haydys correspondence for $3/2$-spinors and non-ellipticity}

\subsection{Haydys correspondence}\label{sec:Haydyscorrespondence}
In this section, we prove the first part of the paper's main result.
\begin{Th}\label{haydyscorrespondence}
    Let $\lp Y,g_Y\rp$ be a Riemannian three manifold with $\SPNC\lp 3\rp$-structure $P_{\SPNC}\rightarrow Y$ and bundle of $3/2$-spinors $\mathscr{W}\rightarrow Y$.
    Then any solution $\lp A,\psi\rp$ of the degenerate Rarita-Schwinger Seiberg-Witten equations
    \[
        Q_A\psi=0,\quad \mu\lp\psi\rp=0
    \]
    gives rise to a solution $\Psi\in\Gamma\lp\mathscr{W}_0\rp$ of the $3/2$-Fueter equation
    \[
        \mathfrak{Q}\Psi=0.
    \]
    Conversely, for any solution $\Psi$ of the $3/2$-Fueter equation, there exists a $\SPNC\lp 3\rp$-structure $P_{\SPNC}\rightarrow Y$ with a connection $A$ on its determinant line bundle along with a section $\psi\in\Gamma\lp\mathscr{W}\rp$ such that the pair $\lp A,\psi\rp$ satisfies the degenerate Rarita-Schwinger Seiberg-Witten equations.
\end{Th}
\begin{proof}
We are interested in a reformulation of the degenerate PDE
\[
    \begin{cases}
        Q_A\psi=0\\
        \mu\left(\psi\right)=0
    \end{cases}
\]
where $\psi\in\Gamma\lp \mathscr{W}\rp$ and $A$ is a hermitian compatible connection for $\text{det}P_{\SPNC\lp 3\rp}$, and $\psi$ is non-vanishing along $Y$. As we have seen above, we know that for each fixed $y\in Y$, the subset 
\[
    \mathscr{M}^{+,\mu}_y:=\mu^{-1}\left(0\right)_y
\]
is a quadratic variety in $\mathscr{M}^+_y= T^*_yY\otimes \slashed{S}_y$ which is smooth away from the origin with smooth component $\mathscr{M}^{\mu}=\mathscr{M}\cap\mu^{-1}\lp 0\rp$. Moreover, the quadratic subvariety
\[
    \mathscr{W}^{+,\mu}_y:=\mathscr{M}^{+,\mu}_y\cap \mathscr{W}_y
\]
is also smooth away from the origin, with smooth component $\mathscr{W}_y=\mathscr{M}^{\mu}_y\cap\mathscr{W}_y$. Both $\mathscr{M}^{\mu}$ and $\mathscr{W}^{\mu}$ are acted on freely by $U(1)$ and away from the origin give a smooth quotient $\mathscr{M}_0:=\mathscr{M}^{\mu}/U(1)$ and $\mathscr{W}_0:=\mathscr{W}^{\mu}/U(1)$. If $\psi\in\Gamma\lp\mathscr{W}\rp $ satisfies the degenerate PDE above, then we have that at $y\in Y$
\[
    \nabla_A\psi_y:T_yY\rightarrow T_y\mathscr{W}^{\mu}.
\]
We can split the vertical distributions of $\mathscr{M}^{\mu}$ and $\mathscr{W}^{\mu}$ into a Hopf component, given by the vertical distribution of the fibration $p:\mathscr{M}^{\mu}\circlearrowleft U(1)\rightarrow \mathscr{M}_0$, and its orthogonal component, which we simply denote by 
\[
    \mathcal{V}\mathscr{M}^{\mu}=\mathcal{L}\oplus\mathcal{H}^M,\quad \mathcal{V}\mathscr{W}^{\mu}=\mathcal{L}\oplus\mathcal{H}^W.
\]
This splitting of $\mathcal{V}\mathscr{W}^{\mu}$ gives us a method for splitting the differential of $\psi$
\[
    \nabla_A\psi_y=\nabla_A\psi^{\mathcal{L}}_y+\nabla_A\psi^{\mathcal{H}^W}_y.
\]
Recall that for a fixed $y\in Y$, the entire discussion is taking place within the $C\ell\left( T^*Y,g_y\right)$-module $\mathscr{M}_y^+$. Moreover, the bundle $\mathcal{H}^M$ is \emph{closed} under Clifford multiplication. The Hopf line bundle $\mathcal{L}$ on $\mathscr{M}^{\mu}$ gets mapped to $\mathcal{V}^{\perp}\mathscr{M}^{\mu}$ under Clifford multiplication, where the orthogonal complement is computed within $\mathcal{V}\mathscr{M}$. Thus, upon applying Clifford contraction to the differential of $\psi$, we obtain the decomposition
\[
    c\circ\nabla_A\psi_y=\lp c\circ\nabla_A\psi_y\rp^{\perp}+\lp c\circ\nabla_A\psi_y\rp^{\mathcal{H}^M}.
\]
Applying the projection on the $3/2$ component gives the evaluation of the Rarita-Schwinger operator
\[
    Q_A\psi_y=\pi_{3/2}\left(c\circ\nabla_A\psi_y\right)^{\perp}+\pi_{3/2}\left( c\circ\nabla_A\psi_y\right)^{\mathcal{H}^M}
\]
This orthogonal projection could ruin the fact that $\left(c\circ\nabla_A\psi\right)^{\perp}$ and $\left(c\circ \nabla_A\psi\right)^{\mathcal{H}^M}$ are linearly independent. Let us see what we can say for sure about the linear dependence of the projections above. For starters, recall that subbundle $\mathcal{N}\leqslant \mathcal{V}\mathscr{M}^{\mu}\big|_{\mathscr{W}^{\mu}}$ has real rank $2$ and when we restrict to $\mathcal{W}^{\mu}$, we obtain a splitting
\[
    \mathcal{V}\mathscr{M}^{\mu}\big|_{\mathscr{W}^{\mu}}=\mathcal{H}^W\oplus \mathcal{H}^{\parallel}\oplus\mathcal{N}\oplus\mathcal{L}.
\]
where $\mathcal{H}^{\parallel}$ is a real rank $2$ line bundle over $\mathscr{W}^{\mu}$. We also have that
\[
    \mathcal{V}\mathscr{W}\big|_{\mathscr{W}^{\mu}}=\mathcal{H}^W\oplus\mathcal{L}\oplus \mathcal{V}^{\perp}\mathscr{W}^{\mu}
\]
where the perpendicular complement is computed relative to the inner product in $\mathcal{V}\mathscr{W}$. In this case we can further split the differential as
\[
    \lp c\circ\nabla_A\psi_y\rp^{\mathcal{H}^M}=\lp c\circ\nabla_A\psi_y\rp^{\mathcal{H}^W}+\lp c\circ\nabla_A\psi_y\rp^{\mathcal{H}^{\parallel}}+\lp c\circ\nabla_A\psi_y\rp^{\mathcal{N}}.
\]
By definition one has that $\pi_{3/2}\lp c\circ\nabla_A\psi_y\rp^{\mathcal{N}}=0$. Moreover, note that if $a\in\mathcal{H}^{\parallel}$ then $\pi_{3/2}a$ must be perpendicular to $\mathcal{H}^W$, for otherwise we could write
\[
    a=\pi_{3/2}a+a^{W^{\perp}}
\]
and this element would have non-trivial inner product with $\mathcal{H}^W$, contradicting the fact that $a$ is perpendicular to $\mathcal{H}^W$. Similarly, we must have that for any $m\in \mathcal{V}^{\perp}\mathscr{M}^{\mu}$, that $\pi_{3/2}m$ must be perpendicular to $\mathcal{H}^W$. What this entails is the following: if $Q_A\psi=0$, then we have
\begin{align*}
    0&=Q_A\psi\\
    &=\pi_{3/2}\lp c\circ\nabla_A\psi\rp^{\mathcal{H}^W}+\pi_{3/2}\lp c\circ\nabla_A\psi\rp^{\mathcal{H}^{\parallel}}+\pi_{3/2}\lp c\circ\nabla_A\psi\rp^{\mathcal{N}}+\pi_{3/2}\lp c\circ\nabla_A\psi\rp^{\perp}\\
    &=\lp c\circ\nabla_A\psi\rp^{\mathcal{H}^W}+\pi_{3/2}\lp c\circ\nabla_A\psi\rp^{\mathcal{H}^{\parallel}}+\pi_{3/2}\lp c\circ\nabla_A\psi\rp^{\perp}
\end{align*}
By what we have just said above, we can assert that we must have
\[
    0=\lp c\circ\nabla_A\psi\rp^{\mathcal{H}^W},\quad 0=\pi_{3/2}\lp c\circ\nabla_A\psi\rp^{\mathcal{H}^{\parallel}}+\pi_{3/2}\lp c\circ\nabla_A\psi\rp^{\perp}
\]
For the section $\Psi=p\psi\in\Gamma\lp\mathscr{W}_0\rp$, the above identity proves
\[
    \mathfrak{Q}\Psi=0.
\]
Thus the first direction of Haydy's correspondence has been established. \par
To obtain the reverse direction of Haydy's correspondence, suppose that we have a section $\Psi\in\Gamma\lp\mathscr{W}_0\rp$ which satisfies
\[
    \mathfrak{Q}\Psi=0.
\]
For the initial $P_{\SPNC\lp 3\rp}$-structure on $Y$, we obtain a principal $U\lp 1\rp$-bundle
\[
    p:\mathscr{W}^{\mu}\circlearrowleft U\lp 1\rp\rightarrow \mathscr{W}_0.
\]
We can pull the principal bundle back along $\Psi$ to obtain a principal $U\lp 1\rp$-bundle on $Y$, $\Psi^*\mathscr{W}^{\mu}\circlearrowleft U\lp 1\rp\rightarrow Y$. The condition for the bundle to be trivial is equivalent to the existence of a section $\psi\in\Gamma\lp\mathscr{W}^{\mu}\rp$ such that $p\psi=\Psi$. If no such section exists, let $L=\Psi^*\mathscr{W}^{\mu}\times_{U\lp 1\rp}\mathbb{C}\rightarrow Y$ be the induced hermitian line bundle. Then for the twisted $\SPNC\lp 3\rp$-structure given by $P_{\SPNC\lp 3\rp}\times_{U\lp 1\rp}\times \Psi^*\mathscr{W}^{\mu}\rightarrow Y$, we obtain a twisted bundle of $3/2$-spinors $\mathscr{W}^L\rightarrow$, and under the diffeomorphism $\mathscr{W}^L_0\simeq \mathscr{W}_0$, the section $\Psi$ of $\mathscr{W}^L_0$ will admits a lifting section $\psi\in\Gamma\lp \mathscr{W}^{L,\mu}\rp$. All this to say, for a given section $\Psi\in\Gamma\lp\mathscr{W}_0\rp$, there always exists a $\SPNC$-structure $P_{\SPNC}\rightarrow Y$ for which the $\Psi$ admits a lifting section $\psi\in\Gamma\lp\mathscr{W}^{\mu}\rp$.\par
If $\psi\in\Gamma\lp\mathscr{W}^{\mu}\rp$ is a lifting section of $\Psi$, then we have that for \emph{any} connection $A$ of $\text{det}P_{\SPNC}$ the identity
\[
    p_*\lp c\circ\nabla_A\psi\rp^{\mathcal{H}^M}=c\circ \nabla^{\mathscr{W}_0}\Psi
\]
where $\nabla^{\mathscr{W}_0}$ is the Ehresmann connection of $\mathscr{W}_0\rightarrow Y$. In this case, if we began with a connection $A_0$ for $L$, note that for any $\alpha\in\Omega^1\lp Y,\mathbb{R}\rp$ one has
\[
    \lp c\circ\nabla_{A_0+i\alpha}\psi\rp^{\perp}=\lp c\circ\nabla_{A_0}\psi\rp^{\perp}+c\lp \alpha\rp i\psi.
\]
The assumption that $\mathfrak{Q}\Psi=0$ only grants us that $\lp c\circ\nabla_A\psi\rp^{\mathcal{H}^W}=0$. To see how to annihilate that $\mathcal{H}^{\parallel}$ and $\perp$ components, we apply $\pi_{3/2}$ to the above relation and note then that we are searching for an $\alpha\in\Omega\lp Y,\mathbb{R}\rp$ satisfying 
\[
    \pi_{3/2}\lp c\circ \nabla_{A_0}\psi\rp^{\mathcal{H}^{\parallel}}+\pi_{3/2}\lp c\circ\nabla_{A_0}\psi\rp^{\perp}=-\pi_{3/2}\lp c\lp\alpha\rp i\psi\rp
\]
By Proposition 2.3 (cite correctly), the restriction of $\pi_{3/2}$ to $\mathcal{V}^{\perp}\mathscr{M}^{\mu}$ is injective and thus surjective onto its image $\mathcal{V}^{\perp}\mathscr{W}^{\mu}$. Moreover, for any $\psi\in\Gamma\lp\mathscr{W}^{\mu}\rp$, the bundle map induced from the following $C^{\infty}\lp Y\rp$-linear map of sections
\begin{align*}
    \Gamma\lp T^*Y\rp&\rightarrow \Gamma\lp\psi^*\mathcal{V}^{\perp}\mathscr{M}^{\mu}\rp\\
    \alpha&\rightarrow c\lp\alpha\rp i\psi
\end{align*}
is an isomorphism. These observations imply that one can indeed find a \emph{unique} $\alpha\in\Omega^1\lp Y,\mathbb{R}\rp$ such that
\[
    \pi_{3/2}\lp c\circ \nabla_{A_0}\psi\rp^{\mathcal{H}^{\parallel}}+\pi_{3/2}\lp c\circ\nabla_{A_0}\psi\rp^{\perp}=-\pi_{3/2}\lp c\lp\alpha\rp i\psi\rp.
\]
This implies that for the connection $A=A_0+i\alpha$, the section $\psi$ satisfies the degenerate Rarita-Schwinger Seiberg-Witten equations
\[
    Q_A\psi=0,\quad \mu\lp\psi\rp=0.
\]
Note that we have shown for each $\psi$ covering $\Psi$, there is a \emph{unique} connection $A$ for $\text{det}P_{Spin^c(3)}$ for which $\lp \psi,A\rp$ satisfy the degenerate Rarita-Schwinger Seiberg-Witten equations. This implies that if we have two solutions $\lp\psi_1,A_1\rp$ and $\lp\psi_2,A_2\rp$ to the degenerate Rarita-Schwinger Seiberg-Witten equations that cover the same section $\Psi$, they must be gauge equivalent, because $\psi_1$ and $\psi_2$ will be related by a gauge transformation. This finishes the reverse direction of the Haydys correspondence.
\end{proof}

\subsection{Non-ellipticity for $3/2$-Fueter operator}\label{sec:ellipticity}
We now turn our attention to the question of nonlinear ellipticity of the Rarita-Schwinger Fueter operator $\mathfrak{Q}$ which leads us to study its linearization at a section $\Psi\in\Gamma\lp\mathscr{W}_0\rp$
\[
    \boldsymbol{\nabla}\mathfrak{Q}_{\Psi}:\Gamma\lp \Psi^*\mathcal{V}\mathscr{W}_0\rp\rightarrow\Gamma\lp \Psi^*\mathcal{V}\mathscr{W}_0\rp.
\]
We first make a few remarks regarding the operator above which will simplify the analysis. Firstly, recall the global formula for $\mathfrak{Q}$, given by
\[
    \mathfrak{Q}{\Psi}=p_0\mathfrak{F}\Psi\in\Gamma\lp Y,\Psi^*\mathcal{V}\mathscr{W}_0\rp.
\]
Recall that for each $y\in Y$ one has $\mathscr{W}_y^{\mu}=\mathscr{M}_y^{\mu}\cap W_y$, however this intersection is not orthogonal, i.e. one does not have that $\mathcal{V}^{\perp}\mathscr{M}^{\mu}$ is contained in $\mathcal{V}\mathscr{W}^{\mu}$. This implies that the orthogonal projection $p_0$ from $\mathcal{V}\mathscr{M}\big|_{\mathscr{W}_{0,y}}$ onto $\mathcal{V}\mathscr{W}_{0,y}$ will have a non-trivial covariant derivative to the Levi-Civita connection of $\mathscr{M}_y$, which we denote by $\nabla^{\mathcal{V}}p_0$. Thus, if $\boldsymbol{\nabla}\mathfrak{F}$ denotes the covariant derivative of $\mathfrak{F}$, we have that for $\Psi\in\Gamma\lp \mathscr{W}_0\rp$ and $\Phi\in\Gamma\lp\Psi^*\mathcal{V}\mathscr{W}_0\rp$ that
\[
    \boldsymbol{\nabla}\mathfrak{Q}_{\Psi}\Phi=p_0\boldsymbol{\nabla}\mathfrak{F}_{\Psi}\Phi+\lp\nabla_{\Phi}^{\mathcal{V}}p_0\rp \mathfrak{F}_{\Psi}.
\]
Using our formula for $\boldsymbol{\nabla}\mathfrak{F}$, we can simplify the above expression to
\[
    \boldsymbol{\nabla}\mathfrak{Q}_{\Psi}\Phi=p_0D^{\Psi}\Phi+\lp\nabla^{\mathcal{V}}_{\Phi}p_0\rp\mathfrak{F}_{\Psi}.
\]
Note that the term involving the covariant derivative of $p_0$ is $C^{\infty}\lp Y\rp$-\emph{linear} in $\Phi$ and thus represents a zeroth-order term in the differential operator $\boldsymbol{\nabla}\mathfrak{Q}_{\Psi}$.
\begin{Th}\label{ellipticity}
    For any section $\Psi\in\Gamma\lp\mathscr{W}_0\rp$, the linearization
    \[
        \boldsymbol{\nabla}\mathfrak{Q}_{\Psi}:\Gamma\lp \Psi^*\mathcal{V}\mathscr{W}_0\rp\rightarrow\Gamma\lp \Psi^*\mathcal{V}\mathscr{W}_0\rp
    \]
    is a real non-elliptic differential operator of order $1$.
\end{Th}
\begin{proof}
Fix $\Psi\in\Gamma\lp Y,\mathscr{W}_0\rp$ and recall that within the formula 
\[
    \boldsymbol{\nabla}\mathfrak{Q}_{\Psi}\Phi=p_0D^{\Psi}\Phi+\lp\nabla^{\mathcal{V}}_{\Phi}p_0\rp\mathfrak{F}_{\Psi}
\]
where $D^{\Psi}:\Gamma\lp Y,\Psi^*\mathcal{V}\mathscr{M}_0\rp\rightarrow\Gamma\lp Y,\Psi^*\mathcal{V}\mathscr{M}_0\rp$ is the Dirac operator associated with the Clifford module $\Psi^*\mathcal{V}\mathscr{M}_0$ and $p_0$ is the orthogonal projection from $\mathcal{V}\mathscr{M}_0$ onto $\mathcal{V}\mathscr{W}_0$. This implies that the symbol of $\boldsymbol{\nabla}\mathfrak{Q}_{\Psi}$ at $y\in Y$ and $\xi\in T^*Y$ will be
\[
    \sigma\lp\boldsymbol{\nabla}\mathfrak{Q}_{\Psi}\rp_y\lp\xi\rp=i\lp\Psi^*p_0\rp c^{\mathcal{V}\mathscr{M}_0}\lp\xi\rp.
\]
Our task is to show that for any choice of $\Psi\in\Gamma\lp Y,\mathscr{W}_0\rp$, there is $\xi\neq 0$ such that the above operator is not invertible.\par
As in the proof of Haydys' correspondence, we are free to assume that there is a section $\psi\in\Gamma\lp Y,\mathscr{W}^{\mu}\rp$ which covers $\Psi$. In this case the Clifford module $\Psi^*\mathcal{V}\mathscr{M}_0=\mathcal{H}^M$ is a subbmodule of the Clifford module $\psi^*T^*Y\otimes\slashed{S}\simeq T^*Y\otimes\slashed{S}$. It is here that we need to utilize the background $U(2)$-symmetry. Namely, keep in mind that when verifying ellipticity, it is a pointwise condition one checks at each $y\in Y$. Thus we can assume that coordinates have been fixed and we have identified $\mathscr{M}^+_y\simeq\mathfrak{sl}\lp2,\mathbb{C}\rp\otimes_{\mathbb{C}}\mathbb{C}^2$ and $T^*_yY\simeq\mathfrak{su}\lp 2\rp$. Furthermore the subbundle $\mathcal{H}^W\rightarrow\mathscr{W}^{\mu}$ is $U(2)$-equivariant and is identified with $\mathcal{V}\mathscr{W}_0$ under the projection $p:\mathscr{W}^{\mu}\rightarrow\mathscr{W}_0$. In this case we note that the Clifford multiplication is $U(2)$-equivariant, i.e. for any $U\in U(2)$ with image $U_{SO}\in SO(3)$ under the natural representation $U(2)\simeq Spin^c\lp 3\rp\rightarrow SO(3)$, one has
\[
    U\cdot c\lp\xi\rp w=c\lp U_{SO}\xi\rp U\cdot w,\quad w\in \mathfrak{sl}\lp 2,\mathbb{C}\rp\otimes_{\mathbb{C}}\mathbb{C}^2,
\]
where $w\in\mathcal{H}^W_{\psi\lp y\rp}$ and $U\cdot w\in\mathcal{H}^W_{U\cdot\psi\lp y\rp}$. Therefore, we are free to move the point $\psi\lp y\rp$ to a well-chosen point in $\mathscr{W}^{\mu}_y$ which will make the computations feasible. With this in mind, as before, we pick $y\in Y$ so that
\[
    \psi_0=\psi\lp y\rp=\lp\frac{1}{\sqrt{2}},0,0,0,\frac{1}{\sqrt{2}},0\rp.
\]
At $y$, the Clifford algebra $C\ell\lp T^*_yY,g_y\rp$ is generated by $\{1\otimes E_1,1\otimes E_2,1\otimes E_3\}$ and a dimension count will show that the submodule $\mathcal{H}^M$ will be rank $2$ and thus will be generated by two well-chosen vectors. Recall that at our chosen point, one has that 
\[
    \text{Ker}D\mu=\mathbb{R}\cdot\left\{ e_3,e_4,e_6,e_2-e_7,e_1+e_8,e_9,e_5+e_{10},e_{11},e_{12}\right\}
\]
and the Hopf line $\mathcal{L}$ is spanned by $e_4+e_{11}$. Thus, if we pick $e_6=H\otimes i\vec{e}_1$ and $e_5+e_{10}=Y\otimes i\vec{e}_1+X\otimes i\vec{e}_2$, these will both lie in $\mathcal{H}^M=\text{Ker}D\mu-\mathcal{L}$. Furthermore, if we study the $C\ell\lp T^*Y,g_Y\rp$-span of these two vectors we obtain two independent submodules of $\mathcal{H}^M$, thus these will be our two well-chosen generators of $\mathcal{H}^M$. We further note that 
\[
    H\otimes i\vec{e}_1=-iE_1\otimes i\vec{e}_1=E_1\otimes \vec{e}_1
\]
thus we see that the submodule spanned by $H\otimes i\vec{e}_1$ is really the submodule spanned by $E_1\otimes \vec{e}_1$. The latter generator $Y\otimes i\vec{e}_1+X\otimes i\vec{e}_2$ cannot be written as $v\otimes s$ for some vector $v\in T_yY$ and $s\in\mathbb{H}$ and in fact no nonzero vectors in this submodule can be expressed in this way. We find then the following four vectors  give a $\mathbb{R}$-spanning set for $\text{Ker}m$
\begin{align*}
    &Y\otimes i\vec{e}_2-X\otimes i\vec{e}_1,\quad Y\otimes\vec{e}_2+X\otimes\vec{e}_1,\\
    &Y\otimes i\vec{e}_1+X\otimes i\vec{e}_2-H\otimes i\vec{e}_1+H\otimes i\vec{e}_2,\quad -Y\otimes\vec{e}_1+X\otimes\vec{e}_2-H\otimes\vec{e}_1-H\otimes\vec{e}_2.
\end{align*}
and note that although the above vectors are orthogonal to one another, their lengths are given by $1$, $1$, $\sqrt{3}$, and $\sqrt{3}$, respectively (recall that we are dividing the trace norm on $\mathfrak{sl}\lp 2,\mathbb{C}\rp$ by $2$). We use $\mathscr{B}$ to denote the orthonormal basis given by the vectors above divided by their lengths. For a number $\alpha\in\mathbb{R}$ and $\beta=\beta_2+i\beta_3\in\mathbb{C}$ we consider elements in the Clifford algebra of the form
\[
    \alpha E_1+\beta_2E_2+\beta_3E_3=\begin{bmatrix}
        i\alpha & -\overline{\beta} \\
        \beta & -i\alpha
    \end{bmatrix}
\]
Thus, we compute 
\begin{align*}
    \lp \alpha E_1+\beta_2E_2+\beta_3E_3\rp &\lp Y\otimes i\vec{e}_2-X\otimes i\vec{e}_1\rp =Y\otimes \lp -i\overline{\beta}\vec{e}_1+\alpha\vec{e}_2\rp-X\otimes\lp -\alpha\vec{e}_1+i\beta\vec{e}_2\rp\\
    &=\alpha\lp Y\otimes\vec{e}_2+X\otimes\vec{e}_1\rp-\beta_2\lp Y\otimes i\vec{e}_1+X\otimes i\vec{e}_2\rp+\beta_3\lp -Y\otimes \vec{e}_1+X\otimes\vec{e}_2\rp
\end{align*}
and
\begin{align*}
    \lp \alpha E_1+\beta_2E_2+\beta_3E_3\rp & \lp Y\otimes\vec{e}_2+X\otimes\vec{e}_1\rp=Y\otimes\lp -\overline{\beta}\vec{e}_1-i\alpha\vec{e}_2\rp+X\otimes\lp i\alpha\vec{e}_1+\beta\vec{e}_2\rp\\
    &=-\alpha\lp Y\otimes i\vec{e}_2-X\otimes i\vec{e}_1\rp+\beta_2\lp -Y\otimes\vec{e}_1+X\otimes\vec{e}_2\rp+\beta_3\lp Y\otimes i\vec{e}_1+X\otimes i\vec{e}_2\rp
\end{align*}
and
\begin{align*}
    &\lp \alpha E_1+\beta_2E_2+\beta_3E_3\rp\lp Y\otimes i\vec{e}_1+X\otimes i\vec{e}_2-H\otimes i\vec{e}_1+H\otimes i\vec{e}_2\rp\\
    &=Y\otimes\lp -\alpha\vec{e}_1+i\beta\vec{e}_2\rp+X\otimes\lp -i\overline{\beta}\vec{e}_1+\alpha\vec{e}_2\rp-H\otimes\lp -\alpha\vec{e}_1 +i\beta\vec{e}_2\rp+H\otimes\lp -i\overline{\beta}\vec{e}_1+\alpha\vec{e}_2\rp\\
    &=\beta_2\lp Y\otimes i\vec{e}_2-X\otimes i\vec{e}_1\rp-\beta_3\lp Y\otimes\vec{e}_2+X\otimes\vec{e}_1\rp +\alpha\lp -Y\otimes\vec{e}_1+X\otimes\vec{e}_2+H\otimes\vec{e}_1+H\otimes\vec{e}_2\rp\\
    &\hspace{3in}-H\otimes i\beta\vec{e}_2-H\otimes i\overline{\beta}\vec{e}_1
\end{align*}
We find that the matrix of the symbol of the linearization of the $3/2$-Fueter operator at our well-chosen point is given by
\[
    M\lp \alpha E_1+\beta_2E_2+\beta_3E_3,\mathscr{B} \rp=\begin{bmatrix}
        0 & -\alpha & \frac{\beta_2}{\sqrt{3}} & -\frac{\beta_3}{\sqrt{3}} \\
        \alpha & 0 & -\frac{\beta_3}{\sqrt{3}} & -\frac{\beta_2}{\sqrt{3}}\\
        -\frac{\beta_2}{\sqrt{3}} & \frac{\beta_3}{\sqrt{3}} & 0 & \frac{\alpha}{3} \\
        \frac{\beta_3}{\sqrt{3}} & \frac{\beta_2}{\sqrt{3}} & -\frac{\alpha}{3} & 0
    \end{bmatrix}
\]
To compute the last column we used the fact that Clifford multiplication is skew-symmetric on the euclidean vector space $\mathcal{H}^M$ and composition with an orthogonal projection preserves the skew-symmetry. If we compute the Pfaffian of the skew-symmetric matrix above, which we give the formula below for the reader's convenience
\[
    \text{Pf}\lp A\rp=a_{12}a_{34}-a_{13}a_{24}+a_{14}a_{23}
\]
we get
\[
    \text{Pf}\lp M\lp \alpha E_1+\beta_2E_2+\beta_3E_3,\mathscr{B} \rp\rp=\frac{-\alpha^2+\beta_2^2+\beta_3^2}{3}
\]
which vanishes when $|\alpha|=\sqrt{\beta_2^2+\beta_3^2}$. Recalling that the Pfaffian squares to the determinant, we see that the symbol will not be an isomorphism when $|\alpha|=\sqrt{\beta_2^2+\beta_3^2}$.\par 
\end{proof}

\section{A generalized Rarita-Schwinger-typer operator}\label{sec:generalrsop}

The definition of $\boldsymbol{\nabla}\mathfrak{Q}_{\Psi}$ defined above resembles the definition of a Rarita-Schwinger operator. Thus, borrowing Dirac operator terminology, we incline to call $\boldsymbol{\nabla}\mathfrak{Q}_{\Psi}$ a \textit{Rarita-Schwinger-type operator}. In fact, there are different ways one can construct a Rarita-Schwinger-type operator and $\boldsymbol{\nabla}\mathfrak{Q_{\psi}}$ is a particular instance of a general formalism, which we describe below.

\subsection{Definition and ellipticity}
Let $X$ be a Riemannian manifold. Suppose $\mathscr{E}, \mathscr{S}$ are Clifford bundles over $X$, equipped with their own hermitian inner products. Denote by $D^\mathscr{E} : \Gamma(\mathscr{E}) \to \Gamma(\mathscr{E})$ the Dirac operator on $\mathscr{E}$, and $c_\mathscr{S} : T^*X \otimes \mathscr{S} \to \mathscr{S}$ be the Clifford multiplication on $\mathscr{S}$. Let $h : \mathscr{E} \to T^*X \otimes \mathscr{S}$ be a bundle map. We consider
\[m:= c_\mathscr{S} \circ h : \mathscr{E} \to \mathscr{S},\]
and define $P_m : \mathscr{E} \to \ker\;m$ to be the orthogonal projection of $\mathscr{E}$ onto $\ker\;m$.

\begin{Def}\label{rstype}
    Consider $\mathscr{Q}:= P_m \circ D^{\mathscr{E}}$ to be the differential operator on $\Gamma(\ker\;m)$. We call $\mathscr{Q}$ a \textit{Rarita-Schwinger-type operator}. 
\end{Def}

For our context, we are mainly interested in the (non)ellipticity of the Rarita-Schwinger-type operators. It turns out that not all such operators are elliptic. Hence, it is natural to ask what some sufficient conditions are to ensure that $\mathscr{Q}:= P_m \circ D^\mathscr{E}$ is an elliptic operator.

Note that $\mathscr{Q}$ is elliptic if and only if its principal symbol $\sigma_{\mathscr{Q}}(x,\xi) = P_m(x) \circ c(\xi)|_{\ker\;m_x}$ is injective. Here, $c$ denotes the Clifford multiplication on $\mathscr E$, $x \in X$ and $\xi \in T^*_x X$. This means that, if $v \in \ker\;m_x$ and $P_m(x) \circ c(\xi)v = 0$, then $v = 0$. Equivalently, if $v \in \ker\;m_x$ and 
\[c(\xi)v \in (\ker\; m_x)^\perp = Im\; m^*_x = Im \; h^* \circ c^*_{\mathscr S},\]
then $v$ has to be trivial. Here are some sufficient conditions to ensure that $\mathscr{Q}$ is elliptic. In the theorem below, if $0\neq\xi\in T^*X$, we have the contraction map $\iota_{\xi}:T^*_xX\otimes \mathscr{S}\rightarrow\mathscr{S}$ given by taking the inner product of the $T^*_xX$ coefficients with $\xi$.

\begin{Prop}\label{onepossiblecondition}
    Let $\widehat{c}_{\mathscr S}$ denote the Clifford multiplication on the Clifford module $T^*X \otimes \mathscr S$ induced by $1\otimes c_{\mathscr S}$. Suppose the following holds:
    \begin{enumerate}
        \item [(a)] For all $\xi \in T^*X$ and $v \in \mathscr E$, we have $h(c(\xi)v)=\widehat{c}_{\mathscr S}(\xi)h(v)$,
        \item [(b)] $mm^*$ is an isomorphism,
        \item [(c)] For each $\xi \in T^*X$ and $|\xi| =1$, then \[\mathbf{1}_{\mathscr{E}}+2c(\xi)^{-1}m^* (mm^*)^{-1}\iota_\xi h\]
        is injective.
    \end{enumerate}
    Then $\mathscr{Q} = P_m \circ D^{\mathscr E}$ is an elliptic operator.
\end{Prop}

\begin{proof}
    Firstly, if $w \in \ker\;c_{\mathscr S}$, then for each $\xi \in T^*X$ such that $|\xi| = 1$, we must have
    \begin{equation}\label{eq:cliffordmultiplicationandcontraction}
        c_{\mathscr S}(\widehat{c}_{\mathscr S}(\xi)(w)) = -2\iota_{\xi} w.
    \end{equation}
    
    Let $v \in \ker\; m$. Then this means that $h(v) \in \ker\; c_{\mathscr S}$. As a result, by \eqref{eq:cliffordmultiplicationandcontraction}, we have
    \begin{equation}\label{eq:hvandcontraction}
        c_{\mathscr S}(\widehat{c}_{\mathscr S}(\xi) h(v)) = -2\iota_{\xi}h(v).
    \end{equation}

    Additionally, suppose that $c(\xi)v \in Im\; m^*$. This means that there is an $s \in \mathscr S$ such that
    \begin{equation}\label{eq:relationvands}
        c(\xi)v = m^* s.
    \end{equation}
    By applying $m=c_{\mathscr S} \circ h$ to both sides of the above equation, using condition (a) and \eqref{eq:hvandcontraction}, we have
    \begin{align*}
        mm^*(s) & = c_{\mathscr S}(h\circ c(\xi)v)  = c_{\mathscr S}(\widehat{c}_{ \mathscr S}(\xi)h(v)) = -2\iota_{\xi}h(v)
    \end{align*}
    Because of (b), $(mm^*)^{-1}$ exists and we can find what $s$ is in terms of $v$. In particular,
    \begin{equation}\label{eq:formulafors}
        s = -2(mm^*)^{-1}\iota_{\xi}h(v).
    \end{equation}
    Substitute \eqref{eq:formulafors} for $s$ in \eqref{eq:relationvands} to yield
    \[0 = (\mathbf{1}_{\mathscr E} - (-2)c(\xi)^{-1}m^* (mm^*)^{-1}\iota_{\xi}h)(v).\]
    From (c), we see that $\mathbf 1_{E}+2c(\xi)^{-1}m^* (mm^*)^{-1}\iota_\xi h$ is invertible. As a result, $v = 0$. Therefore, $\mathscr{Q} = P_m \circ D^{\mathscr E}$ is elliptic as claimed.
\end{proof}

For example, in the actual Rarita-Schwinger operator setting ($\mathscr{E}:= T^*X \otimes \mathscr{S}$, and $h$ is the identity map), one can check that all three conditions above are satisfied, roughly, because of the fact that all eigenvalues of $(mm^*)^{-1}$ turn out to be $1/3$, which is different from $1/2$. 

Another instance where one can attain ellipticity (without solving an eigenvalue problem) is if one knows that
\begin{equation}\label{eq:normlessthan1}
\norm{-2c(\xi)^{-1}m^* (mm^*)^{-1}\iota_\xi h} < 1.
\end{equation}
Such a situation where \eqref{eq:normlessthan1} holds is the following: let $\mathcal W \subset T^*X$ be any subbundle of $X$ so that $d:=dim\; \mathcal W > 4$ (of course, this requires $dim\;X$ to be appropriately large). Let 
    \[\mathscr E = \mathcal W \otimes \mathscr S,\]
    Note that $\mathscr E$ is a Clifford bundle over $X$. Let $D^\mathscr E$ be the corresponding Dirac operator on $\mathscr E$. Define $m = c_\mathscr S \circ i$, where $i :\mathscr E = \mathcal W \otimes \mathscr S \hookrightarrow T^*X \otimes \mathscr S$ is the inclusion map. Let $P_m$ be the orthogonal projection of $\mathcal W \otimes \mathscr S = \mathscr E$ onto $\ker\; m$. Then
    \[ \mathscr{Q}^\mathcal W = P_m \circ D^\mathscr E\]
    is elliptic.

We shall see why the condition $d > 4$ is relevant for \eqref{eq:normlessthan1}. Indeed, firstly, note that since the Clifford multiplication on $\mathscr E$ is exactly the one on $T^*X \otimes \mathscr S$, which is induced by the Clifford multiplication on $\mathscr S$, $i$ satisfies (a). Next, to check for $(b)$, it suffices to check that $m := c_\mathscr S \circ i: \mathscr{E} \to \mathscr S$ has full rank. Equivalently, for an $s \in \mathscr S$, fix an $w_0 \in \mathcal W$ that is non-trivial, we ask if there is an $s' \in S$ such that $c_\mathfrak s( i(w_0 \otimes s')) = s$. Note that
\begin{align*}
    c_\mathscr{S}(i(w_0\otimes s')) = c_\mathscr S(w_0\otimes s') = c_\mathscr S (w_0)s'.
\end{align*}
Thus, $s'$ can be found explicitly by setting $s' = -c_\mathscr S(w_0)s/||w_0||^2$. Lastly, we consider condition (c). Note that $m^* = i^* \circ c^*_\mathscr S : \mathscr S \to \mathscr E$. Suppose $\mathcal W$ (locally) is spanned by an orthonormal basis $\{w_1,\cdots,w_d\}$. Then
\[m^*(s) = -\sum_{j=1}^{d} w_j \otimes c_\mathscr S(w_j)s.\]
The above implies that
\begin{equation}\label{eq:normofmstar}
    \norm{m^*(s)}^2 = d \norm{s}^2.
\end{equation}
On the other hand, 
\begin{align*}
    mm^*(s) & = m\left(\sum_{j=1}^d w_j \otimes c_\mathscr S (w_j)s \right) = -\sum_{j=1}^d c_\mathscr S (w_j)^2 s = d\cdot s.
\end{align*}
This implies that
\begin{equation}\label{eq:formulaformmstar}
    (mm^*)^{-1}(s) = \dfrac{1}{d}s.
\end{equation}
Now, by \eqref{eq:formulaformmstar},
\begin{align*}
    (mm^*)^{-1}\iota_\xi i(v) = \dfrac{1}{d}\iota_\xi i(v) = \dfrac{1}{d}\iota_\xi(v).
\end{align*}
Thus,
\begin{align*}
    m^*(mm^*)^{-1}\iota_\xi i(v) =  \dfrac{1}{d} m^* \iota_\xi (v).
\end{align*}
As a result, using \eqref{eq:normofmstar} and Cauchy-Schwarz inequality, for $|\xi| = 1$, we must have
\begin{align*}
    \norm{-2c(\xi)^{-1}m^* (mm^*)^{-1}\iota_\xi i(v)} & = \dfrac{2}{d}\norm{c(\xi)^{-1}m^* \iota_\xi (v)} = \dfrac{2}{d} \norm{m^* \iota_\xi(v)} \\
    & = \dfrac{2}{d}\cdot \sqrt{d} \cdot \norm{\iota_\xi (v)} \leq \dfrac{2}{\sqrt{d}}\cdot \norm{v}.
\end{align*}
Hence, for \eqref{eq:normlessthan1} to be satisfied, we need $2/\sqrt{d} < 1$; equivalently, $d > 4$. If that is the case, then (c) is achieved. Therefore, $\mathscr{Q}^\mathcal W$ is elliptic as claimed. 

Conditions (a) and (b) in Proposition \ref{onepossiblecondition} are geometric enough that we expect them to hold for many generic situations. However, condition (c) imposes a strong algebraic constraint. Thus, the verification of (c) potentially varies a lot depending on the situation. Furthermore, we find that the algebraic constraint given by (c) is a feature in the sense that one can always check it by direct computation. Hence, it is of independent interest to see in what situation where (c) is also a necessary condition for ellipticity of Rarita-Schwinger-type operators. The corollary below provide such an instance.

\begin{Cor}\label{remarkaboutellipticityequivalence}
    With the same setting as in Proposition \ref{onepossiblecondition} but now, $\mathscr{E}$ is a sub-Clifford module of $T^*X \otimes \mathscr{S}$. Let $h = \iota$ be the inclusion of $\mathscr{E} \hookrightarrow T^*X \otimes \mathscr S$. Suppose conditions (a)--(b) hold. Assume that
    \[m : \im (c(\xi)|_{\ker\;m}) \to \mathscr{S}\]
    is injective for each $\xi \neq 0$. Then $\mathscr{Q} = P_m \circ D^\mathscr E$ is elliptic if and only if (c) holds.
\end{Cor}

\begin{proof}
    As mentioned above, ellipticity of $\mathscr{Q}$ is equivalent to the fact that if $\xi \in T^*X$ and non-trivial, then $c(\xi) v \in \im m^*$ for $v \in \ker m$ implies that $v = 0$. Without loss of generality, we may assume $\xi$ has unit length. If $m : \im (c(\xi)|_{\ker\;m}) \to \mathscr{S}$ is injective, then
    \[ c(\xi)v = m^* s \Leftrightarrow m(c(\xi)v) = mm^*(s),\]
    which, by computation in the proof of Proposition \ref{onepossiblecondition}, the above is equivalent to 
    \[0 = (\mathbf{1}_{\mathscr E} - (-2)c(\xi)^{-1}m^* (mm^*)^{-1}\iota_{\xi}h)(v),\]
    Therefore, ellipticity of $\mathscr{Q}$ is equivalent to the above equation only has trivial solution as claimed.
\end{proof}

Unfortunately, although the previously discussed example is general and not too computationally heavy (in the sense that one does not have to explicitly calculate the eigenvalues), it is not applicable to our situation of $\boldsymbol{\nabla}\mathfrak{Q}_\Psi$ (one of the obvious reasons is a dimensional one). It turns out that one has to explicitly calculate the eigenvalues of $mm^*$ to prove non-ellipticity. By Proposition \ref{onepossiblecondition} and Corollary \ref{remarkaboutellipticityequivalence}, it is enough to verify that condition (c) fails to obtain non-ellipticity of $\boldsymbol{\nabla}\mathfrak{Q}_\Psi$.

To end this subsection, we give an alternative proof of Theorem \ref{ellipticity} using some of the general formalism developed here. 

\begin{proof}[Another proof of Theorem \ref{ellipticity}]
    We show that $\boldsymbol{\nabla}\mathfrak Q_\Psi$ is non-elliptic by showing that its symbol is not invertible. In fact, we exhibit exactly when the symbol $\sigma(\boldsymbol{\nabla}\mathfrak{Q}_\Psi)_y(\xi)$ is non-invertible. With the formula $\boldsymbol{\nabla}\mathfrak Q = p_0 \boldsymbol{\nabla}\mathfrak F$, and the fact that we have computed the covariant derivative of $\mathfrak F$ above, we have that, up to zeroeth order terms, at $\Psi \in \Gamma(Y, \mathscr{W}_0)$, the identity
    \[\boldsymbol{\nabla}\mathfrak Q_\Psi = \Psi^* p_0 D^\Psi,\]
    where $D^\Psi: \Gamma(Y, \Psi^* \mathcal V\mathscr M_0) \to \Gamma(Y, \Psi^* \mathcal V \mathscr{M}_0)$ is the Dirac operator associated defined previously. 

    Recall that, similar to the proof of Haydys' correspondence, without loss of generality, we suppose there a section $\psi \in \Gamma(Y, \mathscr W^\mu)$ that covers $\Psi$. And the Clifford module $\Psi^* \mathcal V \mathscr{M}_0$ is a sub-Clifford module of $T^* Y \otimes \slashed S$. Consequently, the composition of the inclusion $\Psi^* \mathcal V \mathscr M_0 \to T^* Y \otimes \slashed S$ and the Clifford multiplication $c : T^*Y \otimes \slashed S \to \slashed S$ give us a map $\Psi^* \mathcal{V} \mathscr M_0 \to \slashed S$. Observe that $\Psi^* \mathcal{V} \mathscr W_0$ is the kernel of this map. So, $\boldsymbol{\nabla}\mathfrak Q_\Psi$ is a Rarita-Schwinger-type operator. Thus, by Corollary \ref{remarkaboutellipticityequivalence}, to check for non-ellipticity of the linearization, we go through each part (a)--(c) in Proposition \ref{onepossiblecondition}.

    For part (a), this is trivial. For parts (b)--(c), like in the proof of Theorem \ref{ellipticity} in the previous section, we utilize the $U(2)$-symmetry. The upshot is that, to make computations tractable, we pick $y \in Y$ so that
    \[\psi_0 = \psi(y) = \lp\dfrac{1}{\sqrt 2}, 0, 0, 0, \dfrac{1}{\sqrt 2}, 0 \rp.\]
    Since the subbundle $\mathcal{H}^W \to \mathscr W^\mu$ is $U(2)$-equivariant, at $y \in Y$ and $\xi \in T^*_y Y$, the symbol of $\boldsymbol{\nabla}\mathfrak Q_\Psi$ is given by
    \[\sigma(\boldsymbol{\nabla}\mathfrak Q_\Psi)_y(\xi) = i\Psi^* p_0c(\xi).\]
    Thus, our task has become to show that there is a $\xi \neq 0$ and $w \in \mathcal{V}\mathscr W_{0,\Psi_y}$ non-zero such that $\Psi^* p_0 c(\xi)w = 0$. This question is equivalent to showing that for any $\psi_y \in \mathscr W^\mu_y$ satisfying $p\psi_y = \Psi_y$, and any $w \in \mathcal H^W_{\psi_y}$,that $(c(\xi)w)^{\mathcal H^W_{\psi_y}}\neq 0$. Note that $\mathcal{H}^W = \mathcal{H}^M \cap \mathscr W$, and $\mathscr W = \ker \widehat c$. Hence, if instead one had that $(c(\xi)w)^{\mathcal H^W}=0$, then one would have $c(\xi)w \in P_{\mathcal H^M}\im \widehat c^T$, where $\widehat c^T: \slashed S \to \mathscr M^+$ is the transpose of $\widehat c$, and $P_{\mathcal H^M}: \mathscr M^+ \to \mathcal H^M$ is the orthogonal projection. Because the euclidean inner product on $\mathscr M^+$ is the real part of the natural hermitian inner product on $\mathscr{M}^+$, we have that $\widehat{c}^T = \widehat c^*$, where $\widehat c^*$ is the hermitian adjoint. As a result, if $c\lp\xi\rp w\in \mathcal{H}^M\cap\text{Im}c^* $, there is some $s\in\slashed{S}_y$ such that
\[
    c\lp\xi\rp w=P_{\mathcal{H}^M}\widehat{c}^*s.
\]
Now observe that if $\{e_i\}_{i=1}^3$ is an orthonormal basis of $T_yY$, with dual basis $\{\varphi^i\}_{i=1}^3$, we have $w=\sum_{i=1}^3\varphi^i\otimes s_i$, for some $s_i\in\slashed{S}_y$. In this case for $\xi\in T^*Y_y$ we have
\[
    c\lp\xi\rp w=\sum_{i=1}^3\varphi^i\otimes c\lp \xi\rp s_i\Rightarrow \widehat{c}c\lp\xi\rp w=-2\sum_{i=1}^3\xi_is_i
\]
where $\xi=\xi_i\varphi^i$ and we have used the fact that $w\in \ker\widehat{c}$. We denote the last sum above by $\iota_{\xi}w\in\slashed{S}_y$ and observe that we have found that
\[
    \widehat{c}c\lp\xi\rp w=-2\iota_{\xi}w.
\]
However, if we also have $c\lp\xi\rp w=P_{\mathcal{H}^M}\widehat{c}^*s$, then we obtain 
\[
    \widehat{c}c\lp\xi\rp w=\widehat{c}P_{\mathcal{H}^M}\widehat{c}^*s\Rightarrow -2\iota_{\xi}w=\widehat{c}P_{\mathcal{H}^M}\widehat{c}^*s\Rightarrow -2\lp\widehat{c}P_{\mathcal{H}^M}\widehat{c}^*\rp^{-1}\iota_{\xi}w=s
\]
Reusing the identity $c\lp\xi\rp w=P_{\mathcal{H}^M}\widehat{c}^*s$, we find that 
\[
    c\lp\xi\rp w=-2P_{\mathcal{H}^M}\widehat{c}^*\lp \widehat{c}P_{\mathcal{H}^M}\widehat{c}^*\rp^{-1}\iota_{\xi}w.
\]
Equivalently, this means that
\begin{equation}\label{eq:conditioncfor3halffueter}
    w= -2c(\xi)^{-1}P_{\mathcal{H}^M}\widehat{c}^*\lp \widehat{c}P_{\mathcal{H}^M}\widehat{c}^*\rp^{-1}\iota_{\xi}w.
\end{equation}

Therefore, by Corollary\ref{remarkaboutellipticityequivalence}, to prove the non-ellipticity of $\boldsymbol{\nabla}\mathfrak{Q}_\Psi$, we prove that there is $\xi \neq 0$ and has unit length and $w\neq 0$ satisfying \eqref{eq:conditioncfor3halffueter}. This is verified by Proposition \ref{calculation} below.
\end{proof}

For notational convenience, we set
\[P = P_{\mathcal H^M}, \quad \quad m =\widehat{c}\circ \iota, \quad \quad m^* = P \circ \widehat{c}^*,\]
where $\iota$ denotes the inclusion of the orthogonal complement of the Hopf line component $\mathcal{L}^\perp_{\psi_0}:=\mathscr{E}$ into $\mathfrak{sl}(2,\mathbb C)\otimes_{\mathbb{C}} \mathbb C^2$. The remainder of this section is devoted to the proof of the following propositions.

\begin{Prop}\label{calculation}
    There exists $\xi_0 \in T^*_y Y$ such that the map $\mathbf{1}-2c(\xi_0)\circ P \circ \widehat{c}^* \circ (mm^*)^{-1} \iota_{\xi_0}$ has non-trivial kernel.
\end{Prop}

The application of Corollary \ref{remarkaboutellipticityequivalence} to the proof of Theorem \ref{ellipticity} is unexpected. We explain its relevance in the following proposition.

\begin{Prop}\label{injectivityofm}
One has the following:
\begin{enumerate}
    \item[(a)] If $\xi = \pm E_1$, then \eqref{eq:conditioncfor3halffueter} only has trivial solution.
    \item[(b)] If $\xi \neq \pm E_1$, then $m: \im(c(\xi)|_{\ker m}) \to \slashed{S}$ is injective.
\end{enumerate}
\end{Prop}

Part (a) of Proposition \ref{injectivityofm} implies that the symbol $\sigma(\boldsymbol{\nabla}\mathfrak Q_\Psi)_y(\pm E_1)$ is injective. Thus, to check for non-ellipticity, we have to consider the symbol at $\xi \neq \pm E_1$. Given part (b) of Proposition \ref{injectivityofm}, by Corollary \ref{remarkaboutellipticityequivalence}, non-ellipiticity is equivalent to the fact there is a $\xi \neq \pm E_1$ such that $\mathbf{1}-2c(\xi_0)\circ P \circ \widehat{c}^* \circ (mm^*)^{-1} \iota_{\xi_0}$ has non-trivial kernel. This is confirmed by Proposition \ref{calculation}. And another proof of Theorem \ref{ellipticity} is complete.

\subsection{Proof of Proposition \ref{calculation} and \ref{injectivityofm}}

Recall the convention we used in the paper so far. A basis for $\mathfrak{sl}(2,\mathbb{C})$ is given by
\[X = \begin{bmatrix} 0 & 1 \\ 0& 0\end{bmatrix}, \quad Y = \begin{bmatrix} 0 & 0 \\ 1 & 0 \end{bmatrix}, \quad \widehat{H} = \begin{bmatrix}1/\sqrt{2} & 0 \\ 0 & 1/\sqrt{2} \end{bmatrix} \]
We let $\vec{e}_1$ and $\vec{e}_2$ be the standard orthonormal basis for $\mathbb{C}^2$. An element of $\mathfrak{sl}(2,\mathbb{C})\otimes_{\mathbb{C}}\mathbb{C}^2$ is of the form
\[(z_1X+z_2Y+z_3\widehat{H})\otimes \vec{e}_1 + (w_1X+w_2Y+w_3\widehat{H})\otimes\vec{e}_2.\]
If we view the above element as an element of $\mathbb{C}^3 \times \mathbb{C}^3$, then its coordinates is given by
\[(\vec{z}, \vec{w}) :=(z_1, z_2, z_3, w_1,w_2,w_3) \in \mathbb{C}^3\times \mathbb{C}^3.\]
And as an element of a real $12$-dimensional vector space, its real coordinates are given by the orthonormal basis $\{e_j\}_{j=1}^{12}$, where $e_1$ is a basis element for the real part of $z_1$, and $e_4$ is a basis element of the imaginary part of $z_1$, so on and so forth.

From the proof of Proposition \ref{dimension3/2tperp}, we already know that $ker\;D\mu(1/\sqrt{2}, 0 , 0 , 0 , 1/\sqrt{2}, 0)$ is spanned by
\[ span_{\mathbb{R}}\{e_3, e_4, e_6, (e_2-e_7)/\sqrt{2}, (e_1 + e_8)/\sqrt{2}, e_9, (e_5 + e_{10})/\sqrt{2}, e_{11}, e_{12} \} \]
the Hopf component of $\ker\; D\mu(1/\sqrt{2}, 0 , 0, 0, 1/\sqrt{2},0)$ is spanned by
\[\mathcal{L} = span_{\mathbb{R}}\{(e_4 + e_{11})/\sqrt{2}\}.\]
The orthogonal projection $P_\mathscr E : \ker\; D\mu \to \mathscr E:= \mathcal L^{\perp}$ is given by
\[P_\mathscr E(v) = v - \la v, (e_4+e_{11})/\sqrt{2} \ra\; (e_4 + e_{11})/\sqrt{2}.\]

One has to compose $P_\mathscr E$ with the projection $P_{\ker\; D\mu}: \mathfrak{sl}(2,\mathbb{C})\otimes_{\mathbb{C}} \mathbb{C}^2 \to \ker\; D\mu$ to obtain projection $P: \mathfrak{sl}(2,\mathbb C) \otimes \mathbb{C}^2$. Note that $P_{\ker\; D\mu}$ is the orthogonal projection to the orthogonal complement of the subspace of $\mathfrak{sl}(2,\mathbb{C})\otimes \mathbb{C}^2$ spanned by
\[X \otimes E_1 \;\dfrac{i}{\sqrt{2}}\vec{e}_1 + \color{black}{Y} \color{black} \otimes E_1\;\dfrac{i}{\sqrt{2}}\vec{e}_2 \leftrightarrow \dfrac{e_8-e_1}{\sqrt{2}},\]
\[X \otimes E_2 \;\dfrac{i}{\sqrt{2}}\vec{e}_1 + \color{black}Y \color{black} \otimes E_2\;\dfrac{i}{\sqrt{2}}\vec{e}_2\leftrightarrow \dfrac{e_{10}-e_5}{\sqrt{2}},\]
\[X \otimes E_3 \;\dfrac{i}{\sqrt{2}}\vec{e}_1 + \color{black}Y \color{black} \otimes E_3\;\dfrac{i}{\sqrt{2}}\vec{e}_2 \leftrightarrow -\dfrac{e_2+e_7}{\sqrt{2}},\]
where $E_j$ are the Pauli spin matrices defined previously. Thus, for $v \in \mathfrak{sl}(2,\mathbb{C})\otimes \mathbb{C}^2$, written as linear combinations of $\{e_1,\cdots,e_{12}\}$ as
\[v = \sum_{j=1}^{12} a_j e_j,\]
we have $P_{\ker\; D\mu}(v)$ described explicitly as a vector of $\mathbb{C}^3 \times \mathbb{C}^3$ as 
\[\left(\dfrac{a_1+a_8}{2}+ia_4, \dfrac{a_2-a_7}{2}+i\dfrac{a_5+a_{10}}{2},a_3 + i a_6, \dfrac{a_7-a_2}{2}+i\dfrac{a_5+a_{10}}{2}, \dfrac{a_1+a_8}{2} + ia_{11}, a_9 + i a_{12} \right).\]
As a result, in real coordinates, the formula for the orthogonal projection $P= P_\mathscr E \circ P_{\ker\;D\mu}$ is given by
\begin{align}\label{eq:projectionP}
P(a_1,\cdots,a_{12}) = 
\end{align}
\begin{align*}
={}&\biggl(
\frac{a_1+a_8}{2}
+i\frac{a_4-a_{11}}{2},
\frac{a_2-a_7}{2}
+i\frac{a_5+a_{10}}{2},
a_3+ia_6,
\\
&\qquad
\frac{a_7-a_2}{2}
+i\frac{a_5+a_{10}}{2},
\frac{a_8+a_1}{2}
+i\frac{a_{11}-a_4}{2},
a_9+ia_{12}
\biggr).
\end{align*}

Let 
$$s = \begin{bmatrix} a \\ b \end{bmatrix} \in \mathbb C^2,$$
where the notation is used to mean that $s  = a\vec{e}_1 + b \vec{e}_2$. We will use these two notations for a vector in $\mathbb C^2$ interchangeably. 

\begin{Lemma}\label{anotherusefullemma}
    If $s = a\vec{e}_1+b\vec{e}_2 \in \mathbb{C}^2$, then
    \[P\circ \widehat{c}^*(s) = X \otimes (a-\overline{b})\vec{e}_2+Y\otimes(b-\overline{a})\vec{e}_1+\widehat{H}\otimes (\sqrt{2} a\vec{e}_1-\sqrt{2} b \vec{e}_2).\]
    As a result, 
    \begin{equation}\label{eq:E1}
2c(E_1)\circ P \circ \widehat{c}^*(s) = X \otimes \begin{bmatrix}0 \\ -2i(a-\overline{b}) \end{bmatrix} + Y \otimes \begin{bmatrix} 2i(b-\overline{a}) \\ 0 \end{bmatrix} + \widehat{H} \otimes \begin{bmatrix} 2\sqrt{2} ia \\ 2\sqrt{2} ib\end{bmatrix},
\end{equation}

\begin{equation}\label{eq:E2}
2c(E_2)\circ P \circ \widehat{c}^*(s) = X \otimes \begin{bmatrix}-2(a-\overline{b}) \\ 0 \end{bmatrix} + Y \otimes \begin{bmatrix} 0 \\ 2(b-\overline{a}) \end{bmatrix} + \widehat{H} \otimes \begin{bmatrix} 2\sqrt{2} b \\ 2\sqrt{2} a\end{bmatrix},
\end{equation}

\begin{equation}\label{eq:E3}
2c(E_3)\circ P \circ \widehat{c}^*(s) = X \otimes \begin{bmatrix}2i(a-\overline{b}) \\ 0 \end{bmatrix} + Y \otimes \begin{bmatrix} 0 \\ 2i(b-\overline{a}) \end{bmatrix} + \widehat{H} \otimes \begin{bmatrix} -2\sqrt{2} b \\ 2\sqrt{2} a\end{bmatrix}.
\end{equation}
\end{Lemma}

\begin{proof}
    Firstly, note that
    \[\widehat{c}^*(s) = -\sum_{j=1}^{3} E_j \otimes E_j (a\vec{e}_1+b\vec{e}_2),\]
    and the fact that $E_1 = \sqrt{2} i \widehat{H}, E_2 = -X + Y, E_3 =i(X+Y)$, we have
    \[\widehat{c}^*(s) = X \otimes 2a \vec{e}_2 + Y\otimes 2b \vec{e}_1 + \widehat{H}\otimes(\sqrt{2} a \vec{e}_1 - \sqrt{2} b \vec{e}_2).\]
    Thus, using \eqref{eq:projectionP}, we have
    \[P\circ \widehat{c}^*(s) = (0, b -\overline{a},\sqrt{2} a , a -\overline{b}, 0, -\sqrt{2} b).\]
    Equivalently,
    \[P\circ \widehat{c}^*(s) = X \otimes (a-\overline{b})\vec{e}_2 + Y \otimes (b-\overline{a})\vec{e}_1 + \widehat{H} \otimes (\sqrt{2} a \vec{e}_1 - \sqrt{2} b \vec{e}_2),\]
    as claimed. The remaining claim follows from a simple calculation. 
\end{proof}

Next, we calculate the eigenvalues of $(mm^*)^{-1}$. Let $s = a\vec{e}_1+b\vec{e}_2 \in S = \mathbb C^2$. By Lemma \ref{anotherusefullemma}, we have
\[m^*(a\vec{e}_1+b\vec{e}_2) = X\otimes (a-\overline b)\vec{e}_2 + Y \otimes (b-\overline a)\vec{e}_1+\widehat{H}\otimes (\sqrt{2} a \vec{e}_1  - \sqrt{2} b \vec{e}_1)\]
Note that
\[X = -\dfrac{E_2+iE_3}{2}, \quad Y = \dfrac{E_2-iE_3}{2}, \quad \widehat{H} = \dfrac{E_1}{\sqrt{2} i}\]
So, we can rewrite $m^*(s)$ as
\begin{align*}
&m^*(a\vec{e}_1+b\vec{e}_2) =\\
&=E_1 \otimes (-ia \vec{e}_1+ib \vec{e}_2) + E_2 \otimes \left(\dfrac{b-\overline a}{2}\vec{e}_1+\dfrac{\overline b -a}{2}\vec{e}_2 \right)+ E_3 \otimes \left(\dfrac{i}{2}(\overline a -b)\vec{e}_1+\dfrac{i}{2}(\overline b - a)\vec{e}_2 \right)
\end{align*}
Thus,
\begin{align*}
    mm^*(a\vec{e}_1+b\vec{e}_2) & =\begin{bmatrix}
        i & 0\\ 0 & -i
    \end{bmatrix}\begin{bmatrix}-ia \\ ib \end{bmatrix} + \begin{bmatrix} 0 & -1 \\ 1 & 0\end{bmatrix}\begin{bmatrix}\dfrac{b-\overline a}{2} \\\dfrac{\overline b -a}{2} \end{bmatrix} + \begin{bmatrix}0 & i \\ i & 0 \end{bmatrix}\begin{bmatrix}\dfrac{i}{2}(\overline a - b) \\ \dfrac{i}{2}(\overline b - a) \end{bmatrix} \\
    & = \begin{bmatrix}a \\b \end{bmatrix} + \begin{bmatrix} \dfrac{a-\overline b}{2} \\ \dfrac{b-\overline a}{2}\end{bmatrix} + \begin{bmatrix} \dfrac{a-\overline b}{2} \\ \dfrac{b-\overline a}{2}\end{bmatrix} = \begin{bmatrix}2a - \overline b \\ 2b - \overline a \end{bmatrix}
\end{align*}
Hence, in real basis, $(mm^*)$ is represented by
\begin{equation*}
    \begin{bmatrix}
        2 & -1 & 0 & 0 \\
        -1 & 2 & 0 & 0 \\
        0 & 0 & 2 & 1 \\
        0 & 0 & 1 & 2
    \end{bmatrix}
\end{equation*}
As a result, in real basis, $(mm^*)^{-1}$ is represented by
\begin{equation*}
    \begin{bmatrix}
        2/3 & 1/3 & 0 & 0 \\
        1/3 & 2/3 & 0 & 0 \\
        0 & 0 & 2/3 & -1/3 \\
        0& 0 & -1/3 & 2/3
    \end{bmatrix}
\end{equation*}
The eigenvalues of the matrix above can be computed easily and they are $\mu_1 = \mu_2 = 1/3$, $ = \lambda_2 = 1$.

Now, let 
\[v = E_1 \otimes s_1 + E_2 \otimes s_2 + E_3 \otimes s_3,\]
where $s_j \in \mathbb C^2$. We decompose $s_j$ according to the eigenspaces of $(mm^*)^{-1}$. In particular, we write
\[ s_j = \begin{bmatrix} a^{(j)} \\ c^{(j)} \end{bmatrix} + \begin{bmatrix} b^{(j)} \\ d^{(j)} \end{bmatrix},\]
where $(a^{(j)}, c^{(j)})^T$ is an eigenvector of $(mm^*)^{-1}$ associated to the eigenvalue $1/3$ and $(b^{(j)}, d^{(j)})^T$ is an eigenvector associated to the eigenvalue $1$. Since we have an explicit formula for $mm^*$ in coordinates, specifically,
\[ mm^* \begin{bmatrix} a \\ b \end{bmatrix} = \begin{bmatrix}2a-\overline{b} \\ 2b- \overline{a} \end{bmatrix},\]
it must be the case that 
\[c^{(j)} = - \overline{a^{(j)}}, \quad \quad \quad d^{(j)} = \overline{b^{(j)}}.\]
Hence, each $s_j$ is only determined by two complex numbers $a^{(j)}$ and $b^{(j)}$. This observation will reduce the complexity in the next calculation significantly. 

Let $\alpha_j \in \mathbb{R}$ so that $\sum \alpha_j^2 = 1$. We consider $\xi = \sum \alpha_j E_j$. Then by the above decomposition, we have
\begin{align*}
    (mm^*)^{-1} \iota_{\xi}v &= \sum_{j=1}^3 \alpha_j \begin{bmatrix}\dfrac{1}{3}a^{(j)}+ b^{(j)}\\-\dfrac{1}{3}\overline{a^(j)}+\overline{b^{(j)}} \end{bmatrix}  = \begin{bmatrix} \sum_{j=1}^3 \alpha_j\left(\dfrac{1}{3}a^{(j)}+b^{(j)} \right)\\ \sum_{j=1}^3 \alpha_j \left(-\dfrac{1}{3}\overline{a^{(j)}}+\overline{b^{(j)}} \right)\end{bmatrix} := \begin{bmatrix}z\\w \end{bmatrix}.
\end{align*}
Thus, by \eqref{eq:E1}, \eqref{eq:E2}, and \eqref{eq:E3}, we must have $2c(\xi)\circ P \circ \widehat{c}^* \circ (mm^*)^{-1}\iota_{\xi}v$ equal to
\[ X \otimes \begin{bmatrix}-2\alpha_2(z-\overline{w})+2\alpha_3 i (z-\overline{w}) \\ -2\alpha_1 i (z-\overline{w}) \end{bmatrix} + Y \otimes \begin{bmatrix} 2\alpha_1 i(w-\overline{z}) \\ 2\alpha_2(w-\overline{z})+2\alpha_3 i (w -\overline{z}) \end{bmatrix} +\] 
\[+\widehat{H} \otimes \begin{bmatrix} 2\sqrt{2} \alpha_1 i z + 2\sqrt{2} \alpha_2 w - 2\sqrt{2} \alpha_3 i w \\ 2\sqrt{2} \alpha_1 i w + 2\sqrt{2} \alpha_2 z + 2\sqrt{2} \alpha_3 i z\end{bmatrix}.\]
This is equivalent to
\[X \otimes -2(z-\overline{w})\begin{bmatrix}\alpha_2 - i\alpha_3 \\ i\alpha_1 \end{bmatrix} + Y \otimes 2(w-\overline{z})\begin{bmatrix}i\alpha_1 \\ \alpha_2 + i\alpha_3 \end{bmatrix} + \widehat{H} \otimes \begin{bmatrix}2\sqrt{2} z i\alpha_1 + 2\sqrt{2} w (\alpha_2 -i\alpha_3) \\ 2\sqrt{2} w i\alpha_1 + 2\sqrt{2}z(\alpha_2 + i\alpha_3) \end{bmatrix}\]
With respect to the $\{E_1, E_2, E_3\}$ basis, we can rewrite the above expression of $2c(\xi)\circ P \circ \widehat{c}^* \circ (mm^*)^{-1}\iota_{\xi}v$ as
\[E_1 \otimes \begin{bmatrix} 2 z\alpha_1 - 2w(\alpha_3 + i\alpha_2) \\ 2w\alpha_1 + 2z(\alpha_3-i\alpha_2) \end{bmatrix} + E_2 \otimes \begin{bmatrix}(z-\overline{w})(\alpha_2 - i\alpha_3) + (w - \overline{z})i\alpha_1 \\ (z-\overline{w}) i\alpha_1 + (w-\overline{z})(\alpha_2+i\alpha_3) \end{bmatrix} + \]
\[+E_3 \otimes \begin{bmatrix} (z-\overline{w})(\alpha_3+i\alpha_2) +(w-\overline{z})\alpha_1 \\ -(z-\overline{w})\alpha_1 + (w-\overline{z})(\alpha_3-i\alpha_2) \end{bmatrix}.\]

With the notations set up above, we are ready to prove Proposition \ref{calculation}

\begin{proof}[Proof of Proposition \ref{calculation}]
    The condition that $v = 2c(\xi)\circ P \circ \widehat{c}^* \circ (mm^*)^{-1}\iota_{\xi}v$, by the above computation, using the same notations, is equivalent to
    \begin{subequations}
\begin{align}
  a^{(1)} + b^{(1)} &= 2 z\alpha_1 - 2w(\alpha_3 + i\alpha_2) \label{eq:first} \\
  -\overline{a^{(1)}}+\overline{b^{(1)}} &= 2w\alpha_1 + 2z(\alpha_3-i\alpha_2)  \label{eq:second} \\
  a^{(2)} + b^{(2)} &= (z-\overline{w})(\alpha_2 - i\alpha_3) + (w - \overline{z})i\alpha_1 \label{eq:third}\\
  -\overline{a^{(2)}}+\overline{b^{(2)}} &= (z-\overline{w}) i\alpha_1 + (w-\overline{z})(\alpha_2+i\alpha_3) \label{eq:fourth}\\
  a^{(3)} + b^{(3)} &=  (z-\overline{w})(\alpha_3+i\alpha_2) +(w-\overline{z})\alpha_1 \label{eq:fifth}\\
  -\overline{a^{(3)}}+\overline{b^{(3)}} &= -(z-\overline{w})\alpha_1 + (w-\overline{z})(\alpha_3-i\alpha_2) \label{eq:sixth}
\end{align}
\end{subequations}

We take \eqref{eq:third} and combine with the conjugate of \eqref{eq:fourth} to obtain
\begin{equation}\label{eq:a2b2}
    b^{(2)}=(w-\overline{z})i\alpha_1, \quad \quad a^{(2)} = (z-\overline{w})(\alpha_2-i\alpha_3).
\end{equation}
Similarly, when combining \eqref{eq:fifth} and the conjugate of \eqref{eq:sixth}, we get
\begin{equation}\label{eq:a3b3}
    b^{(3)}= (w-\overline{z})\alpha_1, \quad \quad a^{(3)}=(z-\overline w)(\alpha_3 +          i\alpha_2).
\end{equation}
As a result, by \eqref{eq:a2b2} and \eqref{eq:a3b3}, we can write
\begin{align*}
    z-\overline w &=\sum_{j=1}^3 \dfrac{2}{3} \alpha_j a^{(j)} \\
    & = \dfrac{2}{3}\alpha_1 a^{(1)} + \dfrac{2}{3}\alpha_2(\alpha_2-i\alpha_3)(z-\overline w) + \dfrac{2}{3}\alpha_3 (\alpha_3 + i\alpha_2)(z-\overline w ).
\end{align*}
The above implies that
\begin{equation}\label{eq:a1}
    2\alpha_1 a^{(1)} = (z-\overline w) (1+2\alpha^2_1).
\end{equation}
Similarly, $z-\overline w$ is re-written as
\begin{align*}
    \overline z + w &= \sum_{j=1}^3 2\alpha_j \overline{b^{(j)}} \\
    & = 2\alpha_1 \overline{b^{(1)}} + 2(\overline w - z) \alpha_1 (\alpha_3 -i\alpha_2),
\end{align*}
which is equivalent to
\begin{equation}\label{eq:b1}
    2\alpha_1 b^{(1)}= (z+\overline w) - 2(w-\overline z)\alpha_1 (\alpha_3 + i\alpha_2).
\end{equation}

On the other hand, we combine \eqref{eq:first} and \eqref{eq:second} to yield
\begin{equation}\label{eq:a1fromsystem}
    2\alpha_1 a^{(1)} = (z-\overline w) 2\alpha^2_1 - 2(w+\overline z) \alpha_1 (\alpha_3 + i\alpha_2).
\end{equation}

\begin{equation}\label{eq:b1fromsystem}
    2\alpha_1 b^{(1)} = (z+\overline w) 2\alpha_1^2 + 2(\overline z -w) \alpha_1 (\alpha_3 + i\alpha_2).
\end{equation}

As a result, by comparing \eqref{eq:a1} with \eqref{eq:a1fromsystem} and \eqref{eq:b1} and \eqref{eq:b1fromsystem}, we have
\begin{subequations}
    \begin{align}
        (z-\overline w) + 2\alpha_1 (\alpha_3 + i\alpha_2) (w+\overline z) = 0 \label{eq:zwfirst}\\
        (2\alpha^2_1 - 1) (z+\overline w) = 0 \label{eq:zwsecond}
    \end{align}
\end{subequations}

From here, we consider two different cases.

{
\renewcommand{\qedsymbol}{}
\begin{proof}[Case 1: $\alpha^2_1 \neq 1/2$]
   In this case, from \eqref{eq:zwsecond}, it is clear that $z + \overline w = 0$. Thus, \eqref{eq:zwfirst} implies that $z-\overline w = 0$. From \eqref{eq:a1}, \eqref{eq:b1}, \eqref{eq:a2b2}, and \eqref{eq:a3b3}, it must be the case that for all $j = 1, 2, 3$,
   \[ a^{(j)} = b^{(j)} = 0.\]
   Equivalently, $v \equiv 0$. 
\end{proof}
}

{
\renewcommand{\qedsymbol}{}
\begin{proof}[Case 2: $\alpha_1 = \pm 1/\sqrt{2}$]
       From \eqref{eq:a1}, we have
       \begin{equation}\label{eq:zwforalphaequal1overroot2}
           z -\overline w = \pm \dfrac{\sqrt{2}}{2}a^{(1)}.
       \end{equation}
       We substitute the expression of \eqref{eq:zwforalphaequal1overroot2} into \eqref{eq:zwfirst} to obtain
       \begin{equation}\label{eq:zwforalphaequal1overroot2second}
           \overline w + z =-\overline {a^{(1)}}\alpha,
       \end{equation}
       where, from this point onward, we set $\alpha := \alpha_3 +i\alpha_2$ and note that $|\alpha| = \sqrt{2}/2$. 

       Now, we substitute both \eqref{eq:zwforalphaequal1overroot2} and \eqref{eq:zwforalphaequal1overroot2second} into \eqref{eq:b1fromsystem} to get
       \begin{align}\label{eq:vanishingb1}
           2\alpha_1 b^{(1)} = -\overline{a^{(1)}}\alpha + 2\left(\pm \dfrac{\sqrt{2}}{2} \overline{a^{(1)}}\right)\left(\pm\dfrac{1}{\sqrt{2}} \right)\alpha = 0,
       \end{align}
       which implies that $b^{(1)} = 0$. 

       With $b^{(1)} = 0$, we use \eqref{eq:first} and \eqref{eq:second} to equate the two expressions for $a^{(1)}$. As a result,
       \begin{equation*}
            \pm 2z\dfrac{1}{\sqrt{2}} - 2w\alpha = -2\overline w \left(\pm \dfrac{1}{\sqrt{2}}\right) + 2 \overline{z} \alpha.
       \end{equation*}
       We substite \eqref{eq:zwforalphaequal1overroot2} and \eqref{eq:zwforalphaequal1overroot2second} into the above to obtain the following relation between $a^{(1)}$ and its conjugate:
       \begin{equation}\label{eq:relationa1andconjugate}
           a^{(1)} = \pm \alpha \sqrt{2} \cdot \overline{a^{(1)}}
       \end{equation}

       To recap, so far, we have $\alpha \in \mathbb C$ such that $|\alpha| = 1/\sqrt{2}$ and
       \[a^{(1)} = \pm \alpha \sqrt{2} \cdot \overline{a^{(1)}}, \quad \quad b^{(1)} = 0.\]
       Furthermore, we can express $a^{(j)}, b^{(j)}$ for $j=2,3$ in terms of $a^{(1)}$ and $\overline{a^{(1)}}$ by substituting \eqref{eq:relationa1andconjugate} into \eqref{eq:a2b2} and \eqref{eq:a3b3}:
       \begin{subequations}
           \begin{align}
               a^{(2)}&=-i\alpha^2\overline {a^{(1)}} \label{eq:newa2}\\
               b^{(2)} &= -\dfrac{i}{2}\overline{a^{(1)}} \label{eq:newb2} \\
               a^{(3)} & = \alpha^2 \overline{a^{(1)}} \label{eq:newa3}\\
               b^{(3)} & = -\dfrac{1}{2}\overline{a^{(1)}} \label{eq:newb3}
           \end{align}
       \end{subequations}
    Thus, by \eqref{eq:newa2}, \eqref{eq:newb2}, \eqref{eq:newa3}, \eqref{eq:newb3}, we can write $v$ as
    \begin{equation}\label{eq:formulaforv}
        v = E_1 \otimes \begin{bmatrix} a^{(1)}\\-\overline{a^{(1)}} \end{bmatrix} + E_2 \otimes \begin{bmatrix} \lp-i\alpha^2-\dfrac{i}{2} \rp \overline{a^{(1)}} \\ \lp-i\overline \alpha^2 +\dfrac{i}{2} \rp a^{(1)} \end{bmatrix} + E_3 \otimes \begin{bmatrix} \alpha^2 \overline{a^{(1)}}-\dfrac{1}{2}\overline{a^{(1)}}\\-\overline{\alpha}^2a^{(1)}-\dfrac{1}{2}a^{(1)} \end{bmatrix}.
    \end{equation}
    Since $v \in \ker\;m$, from the above, we must have 
    \[\begin{bmatrix} i a^{(1)}\\i\cdot \overline{a^{(1)}} \end{bmatrix}+\begin{bmatrix} -\left(-i\overline \alpha^2+\dfrac{i}{2} \right)a^{(1)} \\ \left(-i \alpha^2-\dfrac{i}{2}\right)\overline{a^{(1)}} \end{bmatrix}+\begin{bmatrix} \left(-i\cdot \overline{\alpha}^2-\dfrac{i}{2} \right)a^{(1)} \\ \left(i\alpha^2-\dfrac{i}{2} \right)\overline{a^{(1)}} \end{bmatrix} = 0.\]
\end{proof}
}
The above equation is vacuously true. As a result, for any $\xi_0$ where the $E_1$-component has length-square equal to $1/2$, for $v$ of the form above \eqref{eq:formulaforv}, $v$ is always in the kernel of $\mathbf{1}-2c(\xi_0)\circ P \circ \widehat{c}^* \circ (mm^*)^{-1}\iota_{\xi_0}$.
\end{proof}

Next, we move on to the proof of Proposition \ref{injectivityofm}.

\begin{proof}[Proof of Proposition \ref{injectivityofm} (a)]
    We prove the statement when $\xi = E_1$. For $\xi = -E_1$, the proof is similar.

    Let $v \in \mathscr E$ where $v$ is written as
\[v=X\otimes s_1 + Y \otimes s_2 + \widehat{H}\otimes s_3.\]
Thus, $\iota_{E_1}v$ is given by
\[\iota_{E_1}v = -\dfrac{i}{\sqrt{2}}s_3:=s'_3 = t' + t'^{\perp},\]
where $t'$ is an eigenvector associated with eigenvalue $1/3$ of $(mm^*)^{-1}$ and $t'^\perp $ is an eigenvector associated with eigenvalue $1$. We denote
\[t'=a\vec{e_1}+b\vec{e_2}, \quad t'^\perp=c\vec{e_1}+d\vec{e_2}.\]
So,
\[(mm^*)^{-1}\iota_{E_1}v = \begin{bmatrix}\dfrac{1}{3}a+c \\ \dfrac{1}{3}b+d \end{bmatrix}\]
We apply Lemma \ref{anotherusefullemma} to see that $2c(E_1)\circ P \circ \widehat c^* (mm^*)^{-1}\iota_{E_1}v$ equal to
\[X \otimes -2i \left\{ \left(\dfrac{1}{3}a+c \right) - \left(\dfrac{1}{3}\overline{b}+\overline{d} \right)\right\}\vec{e_2} +Y \otimes 2i\left\{ \left(\dfrac{1}{3}b+d \right) - \left(\dfrac{1}{3}\overline{a}+\overline{c} \right)\right\}\vec{e_1}+ \]
\[+\widehat{H}\otimes 2\sqrt{2} i \left\{\left(\dfrac{1}{3} a+ c+ \lambda_2 f\right)\vec{e_1}-\left(\dfrac{1}{3}b+\lambda_2 d + \lambda_2 g \right)\vec{e_2} \right\}\]
Thus, for $ v = 2c(E_1)\circ P \circ c^*_S (mm^*)^{-1}\iota_{E_1}v$, we need
\begin{equation}\label{eq:anothersystem3}
    \begin{cases}
        s_1 &= -2i \left\{ \left(\dfrac{1}{3}a+c \right) - \left(\dfrac{1}{3}\overline{b}+\overline{d} \right)\right\}\vec{e_2} \\
        s_2 & = 2i\left\{ \left(\dfrac{1}{3}b+d \right) - \left(\dfrac{1}{3}\overline{a}+\overline{c} \right)\right\}\vec{e_1} \\
        a+c & = 2\left(\dfrac{1}{3}a+c \right)\\
        b+d & = 2\left(\dfrac{1}{3}b+d \right)
    \end{cases}
\end{equation}
From the last two equations of \eqref{eq:anothersystem3}, we have
\begin{equation}\label{eq:anothersystem4}
    \begin{cases}
        \dfrac{1}{3}a-c = 0\\
        \dfrac{1}{3}b-d = 0
    \end{cases}
\end{equation}
Taking norm and squaring both sides of \eqref{eq:anothersystem4}, we have
\begin{align}\label{eq:squarenorm1}
        \dfrac{1}{9}|a|^2+|c|^2-\dfrac{2}{3}Re\; a\overline{c}  = 0
\end{align}
\begin{align}\label{eq:squarenorm2}
        \dfrac{1}{9}|b|^2+|d|^2 -\dfrac{2}{3}Re\; b\overline{d} = 0
\end{align}
Combining the two equations above and noting that $a\overline{c} + b \overline{d} = 0$, we obtain
    \[\dfrac{1}{9}|t'|^2 + |t'^\perp|^2 = 0.\]
    As a result, $s_3 = 0$ and $a = b = c = d = 0$, which implies that $s_1 = s_2 =0$. So, $v = 0$ as desired.
\end{proof}

We finish by giving a proof to Proposition \ref{injectivityofm} (b). We are still working at the special point $(1\sqrt{2},0,0,0,1/\sqrt{2},0)$. Recall that the following four vectors give a $\mathbb{R}$-spanning set for $\ker m$
\begin{align*}
    &Y\otimes i\vec{e}_2-X\otimes i\vec{e}_1,\quad Y\otimes\vec{e}_2+X\otimes\vec{e}_1,\\
    &Y\otimes i\vec{e}_1+X\otimes i\vec{e}_2-H\otimes i\vec{e}_1+H\otimes i\vec{e}_2,\quad -Y\otimes\vec{e}_1+X\otimes\vec{e}_2-H\otimes\vec{e}_1-H\otimes\vec{e}_2.
\end{align*}
For a number $\alpha_0\in\mathbb{R}$ and $\beta=\beta_2+i\beta_3\in\mathbb{C}$ we consider elements in the Clifford algebra of the form
\[
    \alpha_0 E_1+\beta_2E_2+\beta_3E_3=\begin{bmatrix}
        i\alpha_0 & -\overline{\beta} \\
        \beta & -i\alpha_0
    \end{bmatrix}.
\]

\begin{proof}[Proof of Proposition \ref{injectivityofm} (b)]
    Fix $\xi \in T^*_y Y$ where $\xi \neq 0$ and $\xi \neq \pm E_1$. Note that then
    \[c(\xi) : \mathscr{E} \to \mathscr{E}\]
    is an isomorphism. Observe that $\ker m \subset \mathscr{E}$. So, we can restrict $c(\xi)|_{\ker m}: \ker m \to \mathscr{E}$. This map is injective. So, by the first isomorphism theorem,
    \[\mathscr{E} \supset \im(c(\xi)|_{\ker m}) \cong \ker m.\]
    Hence, $4=\dim_\mathbb R \ker m = \dim_\mathbb R \im(c(\xi)|_{\ker m})$. Thus, to show that $m: \im(c(\xi)|_{\ker m}) \to \slashed{S}$ is injective, we prove that its matrix representation is invertible.

    By the above discussion, $\ker m$ is $\mathbb{R}$-spanned by
    \begin{align*}
    &Y\otimes i\vec{e}_2-X\otimes i\vec{e}_1,\quad Y\otimes\vec{e}_2+X\otimes\vec{e}_1,\\
    &Y\otimes i\vec{e}_1+X\otimes i\vec{e}_2-H\otimes i\vec{e}_1+H\otimes i\vec{e}_2,\quad -Y\otimes\vec{e}_1+X\otimes\vec{e}_2-H\otimes\vec{e}_1-H\otimes\vec{e}_2.
    \end{align*}
    Furthermore, since $c(\xi)|_{\ker m} : \ker m \to \mathscr{E}$ is injective, applying $c(\xi)$ to each of the basis element of $\ker m$ above will produce for us linearly independent vectors whose $\mathbb{R}$-span gives $\im(c(\xi)|_{\ker m}$. In other words,
    \begin{align*}
    &c(\xi)(Y\otimes i\vec{e}_2-X\otimes i\vec{e}_1),\quad c(\xi)(Y\otimes\vec{e}_2+X\otimes\vec{e}_1),\\
    &c(\xi)(Y\otimes i\vec{e}_1+X\otimes i\vec{e}_2-H\otimes i\vec{e}_1+H\otimes i\vec{e}_2),\quad c(\xi)(-Y\otimes\vec{e}_1+X\otimes\vec{e}_2-H\otimes\vec{e}_1-H\otimes\vec{e}_2).
    \end{align*}
    form a basis for $\im(c(\xi)|_{\ker m}$. Hence, to find a matrix representation of $m: \im(c(\xi)|_{\ker m}) \to \slashed{S} = \mathbb{C}^2$, we compute
    \begin{align*}
    m\circ c(\xi) &\lp Y\otimes i\vec{e}_2-X\otimes i\vec{e}_1\rp =m\lp Y\otimes \lp -i\overline{\beta}\vec{e}_1+\alpha_0\vec{e}_2\rp-X\otimes\lp -\alpha_0\vec{e}_1+i\beta\vec{e}_2\rp\rp = \begin{bmatrix}
        -i\beta \\ -i\overline{\beta}
    \end{bmatrix},
\end{align*}
and
\begin{align*}
    m\circ c(\xi) & \lp Y\otimes\vec{e}_2+X\otimes\vec{e}_1\rp=m\lp Y\otimes\lp -\overline{\beta}\vec{e}_1-i\alpha_0\vec{e}_2\rp+X\otimes\lp i\alpha_0\vec{e}_1+\beta\vec{e}_2\rp \rp= \begin{bmatrix} \beta \\-\overline{\beta}\end{bmatrix},
\end{align*}
and
\begin{align*}
    m\circ c(\xi)& \lp Y\otimes i \vec{e}_1 + X \otimes i\vec{e}_2  + H \otimes (-i\vec{e}_1 + i \vec{e}_2) \rp\\ 
    &= m \lp Y\otimes (-\alpha_0 \vec{e}_1 + i\beta \vec{e}_2) + X \otimes (-i\overline \beta \vec e_1 + \alpha_0 \vec e_2) + H \otimes ((\alpha_0 - i \overline \beta ) \vec e_1 + (\alpha_0-i\beta)\vec e_2) \rp \\
    & = \begin{bmatrix} 2\alpha_0 - i\overline \beta \\ -2\alpha_0 + i\beta\end{bmatrix},
\end{align*}
and
\begin{align*}
    m\circ c(\xi) & \lp-Y\otimes \vec e_1 + X \otimes \vec e_2 + H \otimes (-\vec e_1 - \vec e_2) \rp\\
    & = m\lp Y \otimes (-i\alpha_0 \vec e_1 -\beta \vec e_2) + X \otimes (-\overline \beta \vec e_1-i\alpha_0 \vec e_2) + H \otimes ((-i\alpha_0 + \overline \beta)\vec e_1 + (i\alpha_0 - \beta) \vec e_2) \rp \\
    & = \begin{bmatrix}-2i\alpha_0 + \overline \beta \\ -2i\alpha_0 + \beta \end{bmatrix}.
\end{align*}
As a result, the matrix that represents $m: \im(c(\xi)|_{\ker m}) \to \slashed{S} = \mathbb{C}^2$ with respect to the chosen basis elements above is given by
\[ 
    \begin{bmatrix}
        \beta_3 & \beta_2 & 2\alpha_0 - \beta_3 & \beta_2 \\ 
        -\beta_3 & -\beta_2 & -2\alpha_0 -\beta_3 & \beta_2 \\
        -\beta_2 & \beta_3 & -\beta_2 & -2\alpha_0 -\beta_3\\
        -\beta_2 & \beta_3 & \beta_2 & -2\alpha_0 + \beta_3
    \end{bmatrix}
\]
The above matrix has the determinant given by
\[ 4\beta_2^2+4\beta_3^2+8\beta_2^2\beta_3^2 = 4(\beta_2^2+\beta_3^2)^2,\]
which is always positive because $\beta_2, \beta_3$ cannot be both zero by the hypothesis. The proof is complete.
\end{proof}


\bibliographystyle{plain}

\providecommand{\bysame}{\leavevmode\hbox to3em{\hrulefill}\thinspace}
\providecommand{\MR}{\relax\ifhmode\unskip\space\fi MR }
\providecommand{\MRhref}[2]{%
  \href{http://www.ams.org/mathscinet-getitem?mr=#1}{#2}
}

\bibliography{Reference}

\end{document}